\documentclass[11pt,leqno]{amsart}

\usepackage[top=2.5 cm,bottom=2 cm,left=2.5 cm,right=2 cm]{geometry}
\usepackage[utf8]{inputenc}
  \usepackage{color}

\usepackage{esint}
\usepackage{graphicx}
\usepackage{hyperref}

\newtheorem{example}{Example}[section]
\newtheorem{definition}{Definition}[section]
\newtheorem{theorem}{Theorem}[section]
\newtheorem{lemma}{Lemma}[section]

\newtheorem*{maintheorem*}{Main Theorem}
\allowdisplaybreaks
\numberwithin{equation}{section}

\renewcommand{\i}{\ifmmode\mathit{\mathchar"7010 }\else\char"10 \fi}
\renewcommand{\j}{\ifmmode\mathit{\mathchar"7011 }\else\char"11 \fi}
\newcommand{\R}{\mathbb{R}}
\newcommand{\N}{\mathbb{N}}

\newcommand{\norm}[1]{\left\|#1\right\|}

\newcommand{\weak}{\rightharpoonup}

\newcommand{\pt}{\partial_t}
\newcommand{\ps}{\partial_s}
\newcommand{\ptt}{\partial_{tt}^2}
\newcommand{\px}{\partial_x }
\newcommand{\pxx}{\partial_{xx}^2}

\newcommand{\ptx}{\partial_{tx}^2}

\newcommand{\vfi}{\varphi}
\newcommand{\eps}{\varepsilon}

\def\begi{\begin{itemize}}
\def\endi{\end{itemize}}
\def\bega{\begin{array}}
\def\enda{\end{array}}

\def\bel{\begin{equation}\label}
\def\eeq{\end{equation}}

\newenvironment{Assumptions}
{%

\begin{enumerate}}%
{\end{enumerate}}

{%

\begin{enumerate}}%
{\end{enumerate}}

\begin{document}

\title[Nonlinear waves in adhesive strings]{Nonlinear waves in adhesive strings}

\author{G. M. Coclite}
\address[Giuseppe Maria Coclite]{\newline
  Dipartimento di Matematica, Universit\`a di Bari,
  Via E.~Orabona 4,I--70125 Bari, Italy.}
\email[]{giuseppemaria.coclite@uniba.it}
\urladdr{http://www.dm.uniba.it/Members/coclitegm/}

\author{G. Florio}
\address[Giuseppe Florio]{\newline
  Dipartimento di Meccanica, Matematica e Management, Politecnico di Bari,
  Via E.~Orabona 4,I--70125 Bari, Italy,
  \newline INFN, Sezione di Bari, I--70126 Bari, Italy.}
\email[]{giuseppe.florio@ba.infn.it }
\urladdr{https://sites.google.com/site/giuseppefloriomath/}

\author{M. Ligab\`o}
\address[Marilena Ligab\`o]{\newline
  Dipartimento di Matematica, Universit\`a di Bari,
  Via E.~Orabona 4,I--70125 Bari, Italy.}
\email[]{marilena.ligabo@uniba.it}

\author{F. Maddalena}
\address[Francesco Maddalena]{\newline
  Dipartimento di Meccanica, Matematica e Management, Politecnico di Bari,
  Via E.~Orabona 4,I--70125 Bari, Italy.}
\email[]{francesco.maddalena@poliba.it}
\urladdr{http://www.dimeg.poliba.it/index.php/it/profilo/userprofile/FMadda}

\date{\today}



\thanks{G. M. Coclite and F. Maddalena are members of the Gruppo Nazionale per l'Analisi Matematica, la Probabilit\`a e le loro Applicazioni (GNAMPA) of the Istituto Nazionale di Alta Matematica (INdAM). G. Florio and M. Ligab\`o are supported by the Gruppo Nazionale per la Fisica Matematica  (GNFM) of the Istituto Nazionale di Alta Matematica (INdAM). G> Florio is supported by MIUR through the project VirtualMurgia.}

\begin{abstract} 
We study a 1D semilinear wave equation modeling the dynamic  of an elastic string interacting with a rigid substrate through an adhesive layer. The constitutive law of the adhesive material is assumed elastic up to a finite critical state,
beyond such a value the stress discontinuously drops to zero. Therefore   the semilinear equation is characterized by a source term presenting jump discontinuity. 
Well-posedness of the initial boundary value problem of Neumann type, as well as qualitative properties of the solutions are studied and the evolution of different initial conditions are numerically investigated.
\end{abstract}

\maketitle


\section{Introduction}
\label{sec:intro}
Adhesion, capillarity and wetting phenomena (see \cite{DBQ,KKR}) constitute a challenging arena for mathematical problems due to the complexity of physical mechanisms involved.  A rational understanding in the format of analytical descriptions of such problems,
in addition to being in itself interesting, 
is relevant for both life sciences and manufacturing engineering. In some recent papers (see, e.g., \cite{MP, MPPT, MPT, MPT1}) one of the authors has studied the static problem  of adhesion of elastic thin structures under various constitutive assumptions on the adhesive material. The main goal
of those works relies in characterizing, with the tools of the calculus of variations, the interplay of the occurrence of debonding with other constitutive properties.
The study of the evolution problem related to these physical manifestations require the analysis of multidimensional hyperbolic problems involving mathematical issues not yet well understood.
In this paper we address a prototypical dynamical problem by studying the adhesion of an elastic string glued to a rigid substrate, assuming a discontinuous softening behavior of the adhesive material, i.e. the adhesive stress jumps to zero when a critical value of the displacement is reached.
We consider the mechanical system with the following energy density:
\begin{equation}
\label{def:energydensity}
e[u]=  \frac{1}{2}\rho (\pt u)^2+  \frac{1}{2} K_e (\px u)^2 + \Phi(u),
\end{equation}
where $\rho>0$ denotes the mass density, $K_e$ denotes the elastic stiffness of the string, and $\Phi(u)$ denotes the adhesion potential modeling the energetic contribution of the glue layer. To taking into account the possibility of debonding we assume for the potential $\Phi$ a behavior like in Fig. \ref{fig:phipotential}, for example
\begin{equation}
\label{eq:Phi_intro}
\Phi(u)=\begin{cases}
u^2,&\qquad \text{if $|u|\le u^* $},\\
(u^*)^2,&\qquad \text{if $|u|> u^*$},
\end{cases}
\end{equation}
where $u^*$ denotes the threshold beyond which the glue cannot sustain further stress. 

\begin{figure}[t]
\begin{center}
\includegraphics[height=0.35\textwidth]{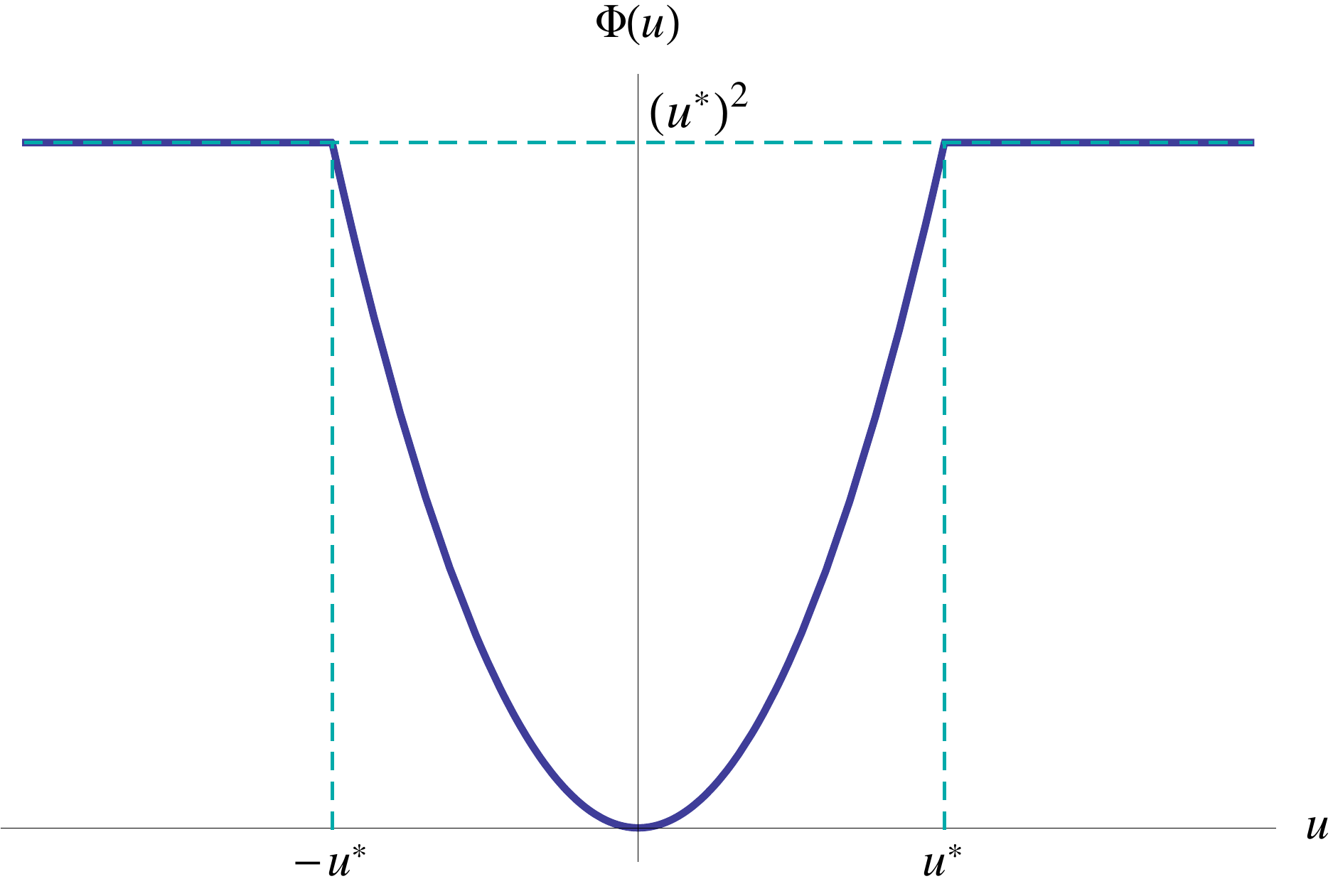} 
\caption{(Color online) Potential $\Phi(u)$ in Eq. \ref{eq:Phi_intro}.
}
\label{fig:phipotential}
\end{center}
\end{figure}

We are interested in the qualitative properties of the Euler equations associated to the above energy density \eqref{def:energydensity} given by
\begin{equation}
\label{eq:e0}
\rho \ptt u=K_e \pxx u-\Phi'\left(u\right),
\end{equation}
equipped with Neumann boundary conditions.

The paper is organized as follows. In Section \ref{sec:pb} we introduce the problem, the main assumptions and the associated energy. In Section \ref{sec:existence} we give the definition of dissipative solution, prove existence (cf. Theorem \ref{th:exist}), regularity (cf. Theorem \ref{th:regularity}), and non-uniqueness for the solutions of initial boundary value problem related to \eqref{eq:e0} (cf. Examples \ref{ex:1}, \ref{ex:2}, \ref{ex:3}). In Section \ref{disc} we focus on the first order formulation of the problem and investigate the interplay between debonding and propagation of singularities along characteristics (cf. Theorem \ref{th:sing}). Finally, in Section \ref{sec:numerics} we consider several initial conditions (in different classes of regularity), numerically investigate the evolutions and highlight peculiar behaviors of the propagation along the characteristics. 

\section{Statement of the problem}
\label{sec:pb}
Let us consider a one dimensional material body, i.e. a {\em string}, whose rest configuration
at the initial time $t=0$ coincides with the interval $[0,L]$ and the displacement field is denoted by 
$$
u:[0,\infty)\times[0,L]\rightarrow \R.
$$
The material is assumed linear elastic and, for sake of notational simplicity, the mass density $\rho$ and the extensional stiffness $K_e$ are assumed both equal to 1. The string interacts with an underlying rigid support through an infinitesimal layer of adhesive material characterized by an internal energy $u\mapsto\Phi(u)$ with the threshold $u^*$ set to 1.

The balance of momentum delivers 
 the initial boundary value problem 
\begin{equation}
\label{eq:el}
\begin{cases}
\ptt u=\pxx u-\Phi'\left(u\right),&\quad t>0,0<x<L,\\
\px u(t,0)=\px u(t,L)=0,&\quad t>0,\\
u(0,x)=u_0(x),&\quad 0<x<L,\\
\pt u(0,x)=u_1(x),&\quad 0<x<L.
\end{cases}
\end{equation}
We shall assume that
\begin{Assumptions}
\item \label{ass:phi} $\Phi\in C(\R)\cap C^1(\R\setminus\{1,-1\})$, $\Phi$ is constant in $(-\infty,-1]$ and in $[1,\infty)$, 
convex in $[-1,1]$, decreasing in $[-1,0]$
 and increasing in $[0,1]$;
\item \label{ass:init} $u_0\in H^1(0,L)$, $u_1\in L^2(0,L)$.
\end{Assumptions}

As a consequence of \ref{ass:phi}, $\Phi'$ has a jump discontinuity in $u=\pm1$ and
\begin{align*}
u\in(-\infty,-1)\cup(1,\infty)&\Rightarrow \Phi'(u)=0,\\
0<u<1&\Rightarrow 0<\Phi'(u)\le \lim_{u\to1^-}\Phi'(u),\\
-1<u<0&\Rightarrow 0>\Phi'(u)\ge \lim_{u\to-1^+}\Phi'(u).
\end{align*}
Assumption \ref{ass:phi}  characterizes the constitutive behavior of the adhesive material, i.e. when
$\vert u\vert =1$ the loss of adhesion manifests through the jump discontinuity of the stress $\Phi'$, hence debonding of the string occurs.

To fix ideas, a function satisfying such assumption is
\begin{equation}
\label{eq:Phi}
\Phi(u)=\begin{cases}
u^2,&\qquad \text{if $|u|\le 1$},\\
1,&\qquad \text{if $|u|> 1$}.
\end{cases}
\end{equation}
In particular we have
\begin{equation}
\label{eq:Phi'}
\Phi'(u)=\begin{cases}
2u,&\qquad \text{if $|u|\le 1$},\\
0,&\qquad \text{if $|u|>1$}.
\end{cases}
\end{equation}

The natural energy associated to the problem \eqref{eq:el} is 
\begin{equation}
\label{en}
E(t)=\int_0^L\left(\frac{(\pt u(t,x))^2+(\px u(t,x))^2}{2}+\Phi(u(t,x))\right)dx.
\end{equation}
Due to the lack of Lipschitz continuity in the nonlinear term $\Phi'$ we cannot expect the existence of conservative solutions, i.e., solutions that 
preserve the energy. This is coherent with the physic behind the problem, when our material is ungluing,
indeed in \cite[Sec. 3.3]{MPPT} the authors describe the hysteresis cycles and  the dissipation associated with the maximum delay strategy corresponding to the 
quasistatic evolution for  a discrete system where the  macroscopic limit (obtained by $\Gamma$-convergence) could be viewed as the system here analyzed. 
Moreover, even mathematically the dissipation of energy is natural. Indeed, when we study the compactness of some approximate solutions
we cannot have bounds on the second derivatives  because we cannot differentiate the equation in \eqref{eq:el}.
Therefore, we have to live with bounds on the first derivatives and then we can have only weak convergence in $H^1$.

\section{Well-posedness and regularity of weak solutions}
\label{sec:existence}

This section is dedicated to the well-posedness  and regularity analysis of \eqref{eq:el}.
We show the existence of Lipshitz continuous dissipative solutions. 
Some examples show that those solutions are not unique and do not depend  continuously on the initial conditions.
Indeed,  in the following section we shall focus on a qualitative analysis  of the discontinuity curves 
of the first derivatives of the solutions. These are the loci where the dissipation of energy occurs. Therefore, it seems quite natural to introduce  the concept of \emph{dissipative solution}:

\begin{definition}
\label{def:sol}
We say that a function $u:[0,\infty)\times[0,L]\to\R$ is a dissipative solution of \eqref{eq:el} if
\begin{itemize}
\item[($i$)] $u\in C([0,\infty)\times[0,L])$;
\item[($ii$)] $\pt u,\,\px u\in L^\infty(0,\infty;L^2(0,L))$;
\item[($iii$)] for every test function $\vfi\in C^\infty(\R^2)$ with compact support
\begin{equation}
\label{eq:weak}
\begin{split}
\int_0^\infty\int_0^L& \left(u\ptt \vfi+\px u\px \vfi+\Phi'\left(u\right)\vfi\right)dtdx\\
&-\int_0^L u_1(x)\vfi(0,x)dx+\int_\R u_0(x)\pt\vfi(0,x)dx=0;
\end{split}
\end{equation}
\item[($iv$)] (energy dissipation) for almost every $t>0$
\begin{equation}
\label{eq:energydissip}
\begin{split}
\int_0^L&\left(\frac{(\pt u(t,x))^2+(\px u(t,x))^2}{2}+\Phi(u(t,x))\right)dx\\
&\qquad\qquad\le \int_0^L\left(\frac{(u_1(x))^2+(u_0'(x))^2}{2}+\Phi(u_0(x))\right)dx.
\end{split}
\end{equation}
\end{itemize}
\end{definition}

\subsection{Existence}
\label{subsec:existence}
The main result of this subsection is the following.

\begin{theorem}[{\bf Existence}]
\label{th:exist}
Let $u_0$ and $u_1$ be given and assume \ref{ass:phi}, \ref{ass:init}. Then \eqref{eq:el} admits a weak solution in the sense of 
Definition \ref{def:sol}.
\end{theorem}

Our argument is based on the approximation of the Neumann problem \eqref{eq:el}
with a sequence of Neumann problems with smooth source terms and smooth initial data.

Let $\{u_{0,n}\}_{n\in\N},\,\{u_{1,n}\}_{n\in\N}\subset C^\infty([0,L]),\,\{\Phi_n\}_{n\in\N}\subset C^\infty(\R)$ be sequences of smooth approximations of 
$u_0,\,u_1$, and $\Phi$ such that

\begin{equation}
\label{eq:assn}
\begin{split}
&u_{0,n}\to u_0\quad\text{in $H^1(0,L)$},\quad u_{1,n}\to u_1\quad\text{in $L^2(0,L)$},\quad \Phi_n\to \Phi \quad \text{uniformly in $\R$},\\
&\Phi_n'\to \Phi' \quad \text{pointwise in $\R$ and uniformly in $\R\setminus\left\{(-1-\eps,-1+\eps)\cup (1-\eps,1+\eps)\right\}$ for every $\eps$},\\
&|u|\ge1+\eps\Rightarrow \Phi_n'(u)=0,\qquad \eps>0,\> n\in\N,\\
&\norm{u_{0,n}}_{H^1(0,L)}\le C,\quad \norm{u_{1,n}}_{L^2(0,L)}\le C, \quad 0\le \Phi_n,\,\Phi_n'\le C,\qquad n\in\N,\\
&u_{0,n}'(0)=u_{0,n}'(L)=u_{1,n}(0)=u_{1,n}(L)=0,\qquad n\in\N,
\end{split}
\end{equation}
where  $C>0$ denotes some constant  independent on $n$.

Let $u_n$ be the unique classical solution of the initial boundary value problem
\begin{equation}
\label{eq:eln}
\begin{cases}
\ptt u_n=\pxx u_n-\Phi_n'(u_n),&\quad t>0,0<x<L,\\
\px u_n(t,0)=\px u_n(t,L)=0,&\quad t>0,\\
u_n(0,x)=u_{0,n}(x),&\quad 0<x<L,\\
\pt u_n(0,x)=u_{1,n}(x),&\quad 0<x<L.
\end{cases}
\end{equation}

\noindent The well-posedness of \eqref{eq:eln} is guaranteed for short time by the Cauchy-Kowaleskaya Theorem \cite{taylor}.
The solutions are indeed global in time thanks to the following a priori estimates.

\begin{lemma}[{\bf Energy conservation}]
\label{lm:energy}
The function 
\begin{equation*}
t\mapsto E_n(t)=\int_0^L\left(\frac{(\pt u_n(t,x))^2+(\px u_n(t,x))^2}{2}+\Phi_n(u_n(t,x))\right)dx
\end{equation*}
is constant for every $n$. In particular, $\{\pt u_n\}_{n\in\N}$ and $\{\px u_n\}_{n\in\N}$ are bounded in $L^\infty(0,\infty;L^2(0,L))$.
\end{lemma}

\begin{proof}
We have that
\begin{align*}
E_n'(t)=&\frac{d}{dt}\int_0^L\left(\frac{(\pt u_n)^2+(\px u_n)^2}{2}+\Phi_n(u_n)\right)dx\\
=&\int_0^L\left(\pt u_n\ptt u_n+\px u_n\ptx u_n+\Phi_n'(u_n)\pt (u_n)\right)dx\\
=&\int_0^L\pt u_n\underbrace{\left(\ptt u_n-\pxx u_n+\Phi_n'(u_n)\right)}_{=0}dx=0.
\end{align*}
\end{proof}

\begin{lemma}[{\bf $L^2$ estimate}]
\label{lm:l2}
The sequence $\{u_n\}_{n\in\N}$ is bounded in $L^\infty(0,T;L^2(0,L))$, for every $T>0$.
\end{lemma}

\begin{proof}
Since
\begin{align*}
\int_0^L u_n^2(t,x)dx=&\int_0^L\left(u_{0,n}(x)+\int_0^t \ps u_n(s,x)ds\right)^2dx\\
\le&2\int_0^L u_{0,n}^2(x)dx+2 \int_0^L\left(\int_0^t |\ps u_n(s,x)|ds\right)^2dx\\
\le&2\int_0^L u_{0,n}^2(x)dx+2t \int_0^t\int_0^L (\ps u_n(s,x))^2dsdx\\
\le&2\int_0^L u_{0,n}^2(x)dx+2t^2 \sup_{s\ge0}\int_0^L (\ps u_n(s,x))^2dx,
\end{align*}
the claim follows from Lemma \ref{lm:energy}.
\end{proof}

\begin{lemma}[{\bf $L^\infty$ estimate}]
\label{lm:linfty}
The sequence $\{u_n\}_{n\in\N}$ is bounded in $L^\infty((0,T)\times(0,L))$, for every $T>0$.
\end{lemma}

\begin{proof}
Fix $0<t<T$ and $0<x<L$.
Lemmas \ref{lm:energy} and \ref{lm:l2} imply that $\{u_n\}_{n\in\N}$ is bounded in $L^\infty(0,T;H^1(0,L))$. 
Since $H^1(0,L)\subset L^\infty(0,L)$ we have
\begin{equation*}
|u_n(t,x)|\le \norm{u_n(t,\cdot)}_{L^\infty(0,L)}\le c\norm{u_n(t,\cdot)}_{H^1(0,L)}
\le c\norm{u_n}_{L^\infty(0,T;H^1(0,L))},
\end{equation*}
for some constant $c>0$ dependeing only on $L$. Therefore 
\begin{equation*}
\norm{u_n}_{L^\infty((0,T)\times(0,L))}\le c\norm{u_n}_{L^\infty(0,T;H^1(0,L))},
\end{equation*}
that gives the claim.
\end{proof}

\begin{proof}[Proof of Theorem \ref{th:exist}]
Thanks to Lemmas \ref{lm:energy}, \ref{lm:l2} and \cite[Theorem 5]{S}
 there exists a function $u$ satisfying ($i$) and ($ii$) of Definition \ref{def:sol} such that,
passing to a subsequence,
\begin{equation}
\label{eq:conv}
\begin{split}
& u_n\weak u \quad \text{in $H^1((0,T)\times(0,L))$, for each $T\ge0$},\\ 
& u_n \to u \quad \text{in $L^\infty((0,T)\times(0,L))$, for each $T\ge0$}.
\end{split}
\end{equation}

We have to verify that $u$ is a weak solution of \eqref{eq:el}.
Let $\vfi\in C^\infty(\R^2)$ be a test function with compact support. From \eqref{eq:eln}, for every $n$ we have
\begin{align*}
\int_0^\infty\int_0^L& \left(u_n\ptt \vfi+\px u_n\px \vfi+\Phi_n'(u_n)\vfi\right)dtdx\\
&-\int_0^L u_{1,n}(x)\vfi(0,x)dx+\int_\R u_{0,n}(x)\pt\vfi(0,x)dx=0.
\end{align*}
As $n\to\infty$, using \eqref{eq:assn} and \eqref{eq:conv}, we get \eqref{eq:weak}.

Finally, \eqref{eq:energydissip} follows from Lemma \ref{lm:energy}, \eqref{eq:assn}, and \eqref{eq:conv}.
\end{proof}

\subsection{Uniqueness}
\label{subsec:uniqueness}
The dissipative solutions of \eqref{eq:el} are not unique.
This is made clear form the following three examples.
In the first example, we show that different regularizations of the discontinuous nonlinear term $\Phi'$ may lead  to 
different dissipative solutions of \eqref{eq:el}.
In the second example, we use only one regularization of $\Phi'$ and approximate the initial conditions in two different ways.
Lastly, the third example shows  that the solutions of \eqref{eq:el} do not continuously depend on the initial data.
Moreover, it seem quite difficult to identify a common asymptotic behavior as $t\to\infty$.

In all the following examples we assume that $\Phi$ is the one defined in \eqref{eq:Phi}.

\begin{example}
\label{ex:1}
Let $\eps>0$. Consider the functions
\begin{align*}
\widetilde\Phi_\eps(u)&=\begin{cases}
u^2,&\quad \text{if $|u|\le 1-\eps,$}\\
\frac{2u-u^2}{\eps}-(1-\eps)\left(\eps+\frac{1}{\eps}\right),&\quad  \text{if $1-\eps\le u\le 1,$}\\
-\frac{2u+u^2}{\eps}-(1-\eps)\left(\eps+\frac{1}{\eps}\right),&\quad  \text{if $-1\le u\le -1+\eps,$}\\
1+\eps^2-\eps,&\quad  \text{if $|u|\ge 1,$}
\end{cases}\\
\overline{\Phi}_\eps(u)&=\begin{cases}
u^2,&\quad  \text{if $|u|\le 1,$}\\
\frac{2(1+\eps)u-u^2}{\eps}-\left(1+\frac{1}{\eps}\right),&\quad  \text{if $1\le u\le 1+\eps,$}\\
-\frac{2(1+\eps)u+u^2}{\eps}-\left(1+\frac{1}{\eps}\right),&\quad  \text{if $-1-\eps\le u\le -1,$}\\
1+\eps,&\quad  \text{if $|u|\ge 1+\eps.$}
\end{cases}
\end{align*}
We have
\begin{align*}
\widetilde\Phi_\eps'(u)&=\begin{cases}
2u,&\quad  \text{if $|u|\le 1-\eps,$}\\
2\frac{1-u}{\eps},&\quad  \text{if $1-\eps\le u\le 1,$}\\
-2\frac{1+u}{\eps},&\quad  \text{if $-1\le u\le -1+\eps,$}\\
0,&\quad  \text{if $|u|\ge 1,$}
\end{cases}\\
\overline{\Phi}_\eps'(u)&=\begin{cases}
2u,&\quad  \text{if $|u|\le 1,$}\\
2\frac{1+\eps-u}{\eps},&\quad  \text{if $1\le u\le 1+\eps,$}\\
-2\frac{1+\eps+u}{\eps},&\quad  \text{if $-1-\eps\le u\le -1,$}\\
0,&\quad  \text{if $|u|\ge 1+\eps.$}
\end{cases}
\end{align*}

The functions 
\begin{equation*}
\widetilde u_\eps(t,x)=1,\qquad \overline{u}_\eps(t,x)=\cos\left(\sqrt{2}\,t\right),
\end{equation*}
solve
\begin{align}
\label{eq:ex.1.1}
&\begin{cases}
\ptt \widetilde u_\eps=\pxx \widetilde u_\eps-\widetilde \Phi_\eps'(\widetilde u_\eps),&\quad t>0,0<x<L,\\
\px \widetilde u_\eps(t,0)=\px \widetilde u_\eps(t,L)=0,&\quad t>0,\\
\widetilde u_\eps(0,x)=1,&\quad 0<x<L,\\
\pt \widetilde u_\eps(0,x)=0,&\quad 0<x<L,
\end{cases}
\\&
\label{eq:ex.1.2}
\begin{cases}
\ptt \overline{u}_\eps=\pxx \overline{u}_\eps-\overline{\Phi}_\eps'(\overline{u}_\eps),&\quad t>0,0<x<L,\\
\px \overline{u}_\eps(t,0)=\px \overline{u}_\eps(t,L)=0,&\quad t>0,\\
\overline{u}_\eps(0,x)=1,&\quad 0<x<L,\\
\pt \overline{u}_\eps(0,x)=0,&\quad 0<x<L.
\end{cases}
\end{align}
As $\eps\to0$ we have
\begin{equation*}
\widetilde u_\eps(t,x)\to \widetilde u(t,x)=1,\qquad \overline{u}_\eps(t,x)\to \overline{u}(t,x)=\cos\left(\sqrt{2}\,t\right),
\end{equation*}
and $\widetilde u$ and $\overline{u}$ provide two different solutions of \eqref{eq:el} in correspondence of the initial data
\begin{equation*}
u_0(x)=1,\qquad u_1(x)=0.
\end{equation*}

The energies associated to \eqref{eq:ex.1.1} and \eqref{eq:ex.1.2} are 
\begin{align*}
\widetilde E_\eps(t)&=\int_0^L\left(\frac{(\pt \widetilde u_\eps(t,x))^2+(\px \widetilde u_\eps(t,x))^2}{2}+\widetilde \Phi_\eps(\widetilde u_\eps(t,x))\right)dx=(1+\eps^2-\eps)L,\\
\overline{E}_\eps(t)&=\int_0^L\left(\frac{(\pt \overline{u}_\eps(t,x))^2+(\px \overline{u}_\eps(t,x))^2}{2}+\overline{\Phi}_\eps(\overline{u}_\eps(t,x))\right)dx=L,
\end{align*}
respectively.
\end{example}

\begin{example}
\label{ex:2}
Let $\eps>0$. Consider the function
\begin{equation*}
\Phi_\eps(u)=\begin{cases}
\frac{2-\eps}{2}u^2,&\quad  \text{if $|u|\le 1,$}\\
\frac{2-\eps}{\eps}\left((1+\eps)\left(u-\frac{1}{2}\right)-\frac{u^2}{2}\right),&\quad  \text{if $1\le u\le 1+\eps,$}\\
\frac{\eps-2}{\eps}\left((1+\eps)\left(u+\frac{1}{2}\right)+\frac{u^2}{2}\right),&\quad  \text{if $-1-\eps\le u\le -1,$}\\
\frac{(2-\eps)(1+\eps)}{2},&\quad  \text{if $ |u|\ge 1+\eps.$}
\end{cases}
\end{equation*}
We have
\begin{equation*}
\Phi_\eps'(u)=\begin{cases}
(2-\eps)u,&\quad  \text{if $|u|\le 1,$}\\
\frac{2-\eps}{\eps}(1+\eps-u),&\quad  \text{if $1\le u\le 1+\eps,$}\\
\frac{\eps-2}{\eps}(1+\eps+u),&\quad  \text{if $-1-\eps\le u\le -1,$}\\
0,&\quad  \text{if $|u|\ge 1+\eps.$}
\end{cases}
\end{equation*}
The functions 
\begin{equation*}
u_\eps(t,x)=(1-\eps)\cos\left(\sqrt{2-\eps}\,t\right),\qquad v_\eps(t,x)=1+\eps
\end{equation*}
solve
\begin{align}
\label{eq:ex.2.1}
&\begin{cases}
\ptt u_\eps=\pxx u_\eps-\Phi_\eps'(u_\eps),&\quad t>0,0<x<L,\\
\px u_\eps(t,0)=\px u_\eps(t,L)=0,&\quad t>0,\\
u_\eps(0,x)=1-\eps,&\quad 0<x<L,\\
\pt u_\eps(0,x)=0,&\quad 0<x<L,
\end{cases}
\\&
\label{eq:ex.2.2}
\begin{cases}
\ptt v_\eps=\pxx v_\eps-\Phi_\eps'(v_\eps),&\quad t>0,0<x<L,\\
\px v_\eps(t,0)=\px v_\eps(t,L)=0,&\quad t>0,\\
v_\eps(0,x)=1+\eps,&\quad 0<x<L,\\
\pt v_\eps(0,x)=0,&\quad 0<x<L.
\end{cases}
\end{align}
As $\eps\to0$ we have
\begin{equation*}
u_\eps(t,x)\to u(t,x)=\cos\left(\sqrt{2}\,t\right),\qquad v_\eps(t,x)\to v(t,x)=1,
\end{equation*}
and $u$ and $v$ provides two different solutions of \eqref{eq:el} in correspondence of the initial data
\begin{equation*}
u_0(x)=1,\qquad u_1(x)=0.
\end{equation*}
The energies associated to \eqref{eq:ex.2.1} and \eqref{eq:ex.2.2} are 
\begin{align*}
E_\eps(t)&=\int_0^L\left(\frac{(\pt u_\eps(t,x))^2+(\px u_\eps(t,x))^2}{2}+ \Phi_\eps(u_\eps(t,x))\right)dx=\frac{(2-\eps)(1-\eps)^2}{2}L,\\
\mathcal{E}_\eps(t)&=\int_0^L\left(\frac{(\pt v_\eps(t,x))^2+(\px v_\eps(t,x))^2}{2}+\Phi_\eps(v_\eps(t,x))\right)dx=\frac{(2-\eps)(1+\eps)}{2}L,
\end{align*}
respectively.
\end{example}

\begin{example}
\label{ex:3}
For every $\eps>0$, the solutions $u_\eps$ and $v_\eps$ of the two following problems
\begin{align}
\label{eq:ex.3.1}
&\begin{cases}
\ptt u_\eps=\pxx u_\eps-\Phi'(u_\eps),&\quad t>0,0<x<L,\\
\px u_\eps(t,0)=\px u_\eps(t,L)=0,&\quad t>0,\\
u_\eps(0,x)=1+\eps,&\quad 0<x<L,\\
\pt u_\eps(0,x)=\eps,&\quad 0<x<L,
\end{cases}
\\&
\label{eq:ex.3.2}
\begin{cases}
\ptt v_\eps=\pxx v_\eps-\Phi'(v_\eps),&\quad t>0,0<x<L,\\
\px v_\eps(t,0)=\px v_\eps(t,L)=0,&\quad t>0,\\
v_\eps(0,x)=1-\eps,&\quad 0<x<L,\\
\pt v_\eps(0,x)=0,&\quad 0<x<L,
\end{cases}
\end{align}
are
\begin{equation*}
u_\eps(t,x)=\eps t+1+\eps,\qquad 
v_\eps(t,x)=(1-\eps)\cos(\sqrt{2}t).
\end{equation*}
We have
\begin{align*}
&\norm{u_\eps(0,\cdot)-v_\eps(0,\cdot)}_{L^2(0,L)}+\norm{\pt u_\eps(0,\cdot)-\pt v_\eps(0,\cdot)}_{L^2(0,L)}=3\eps\sqrt{L},\\
&\lim_{t\to\infty}u_\eps(t,x)=\infty,\qquad \limsup_{t\to\infty}v_\eps(t,x)=1-\eps.
\end{align*}
Moreover, as $\eps\to0$,
\begin{equation*}
u_\eps(t,x)\to 1,\qquad 
v_\eps(t,x)\to \cos(\sqrt{2}t).
\end{equation*}
The energies associated to \eqref{eq:ex.3.1} and \eqref{eq:ex.3.2} are 
\begin{align*}
E_\eps(t)&=\int_0^L\left(\frac{(\pt u_\eps(t,x))^2+(\px u_\eps(t,x))^2}{2}+ \Phi(u_\eps(t,x))\right)dx=\frac{\eps^2+2}{2}L,\\
\mathcal{E}_\eps(t)&=\int_0^L\left(\frac{(\pt v_\eps(t,x))^2+(\px v_\eps(t,x))^2}{2}+\Phi(v_\eps(t,x))\right)dx=(1-\eps)^2L,
\end{align*}
respectively.
\end{example}

\subsection{Regularity}
\label{subsec:reg}

This subsection is devoted to the maximal regularity we can expect for the dissipative solutions of \eqref{eq:el}.
We show that if the $t$ and $x$ derivative of the solutions at time $t=0$ are bounded then we have 
locally Lipshitz continuous solutions. In the following section, using a first order formulation of \eqref{eq:el} we will  show that we cannot expct more regularity even if we
consider more regular inital data. 

\begin{theorem}
\label{th:regularity}
Let $u_0$ and $u_1$ be given and assume \ref{ass:phi}, \ref{ass:init}.
If $u$ is a dissipative solution of \eqref{eq:el} and 
\begin{equation}
\label{eq:regass}
u_0\in W^{1,\infty}(0,L),\qquad u_1\in L^{\infty}(0,L),
\end{equation}
then
\begin{equation}
\label{eq:regclaim}
u\in C([0,\infty)\times[0,L])\cap W^{1,\infty}((0,T)\times(0,L)),
\end{equation}
for every $T>0$.
\end{theorem}

\begin{proof}
Let $u$ be a solution of \eqref{eq:el}.
Consider the function $\widetilde u:[0,\infty)\times\R\to\R$ defined as the $2L-$periodic (in space) extension of the function $(t,x)\in[0,\infty)\times[-L,L]\mapsto u(t,|x|)$.
$\widetilde u$ is the unique solution of the Cauchy Problem
\begin{equation}
\label{eq:elu}
\begin{cases}
\ptt v=\pxx v-\Phi'(\widetilde u),&\quad t>0,x\in\R,\\
v(0,x)=\widetilde u_0(x),&\quad x\in\R,\\
\pt v(0,x)=\widetilde u_1(x),&\quad x\in\R,
\end{cases}
\end{equation}
where $\widetilde u_0$ and $\widetilde u_1$ are the $2L-$periodic  extensions of the functions $x\in[-L,L]\mapsto u_0(|x|)$ and $x\in[-L,L]\mapsto u_1(|x|)$, 
respectively.
Therefore, the following representation formula holds
\begin{equation}
\label{eq:u}
\widetilde u(t,x)=\frac{\widetilde u_0(x+t)+\widetilde u_0(x-t)}{2}+\frac{1}{2}\int_{x-t}^{x+t}\widetilde u_1(s)ds
+\frac{1}{2}\int_0^t\int_{x-(t-s)}^{x+(t-s)}\Phi'(\widetilde u(s,y))dsdy.
\end{equation}

We have that
\begin{align*}
|\widetilde u(t,x)&-\widetilde u(t',x')|\\
=&\frac{|\widetilde u_0(x+t)-\widetilde u_0(x'+t')|+|\widetilde u_0(x-t)-\widetilde u_0(x'-t')|}{2}\\
&+\frac{1}{2}\left|\int_{x-t}^{x+t}\widetilde u_1(s)ds-\int_{x'-t}^{x'+t}\widetilde u_1(s)ds\right|
+\frac{1}{2}\left|\int_{x'-t}^{x'+t}\widetilde u_1(s)ds-\int_{x'-t'}^{x'+t'}\widetilde u_1(s)ds\right|\\
&+\frac{1}{2}\left|\int_0^t\int_{x-(t-s)}^{x+(t-s)}\Phi'(\widetilde u(s,y))dsdy-\int_0^{t'}\int_{x-(t-s)}^{x+(t-s)}\Phi'(\widetilde u(s,y))dsdy\right|\\
&+\frac{1}{2}\left|\int_0^{t'}\int_{x-(t-s)}^{x+(t-s)}\Phi'(\widetilde u(s,y))dsdy-\int_0^{t'}\int_{x-(t'-s)}^{x+(t'-s)}\Phi'(\widetilde u(s,y))dsdy\right|\\
&+\frac{1}{2}\left|\int_0^{t'}\int_{x-(t'-s)}^{x+(t'-s)}\Phi'(\widetilde u(s,y))dsdy-\int_0^{t'}\int_{x'-(t'-s)}^{x'+(t'-s)}\Phi'(\widetilde u(s,y))dsdy\right|\\
\le&\left(\norm{\widetilde u_0'}_{L^\infty(\R)}+\frac{\norm{\widetilde u_1}_{L^\infty(\R)}}{2}+\frac{3}{2}\norm{\Phi'}_{L^\infty(\R)}(t+t')\right)\left(|x-x'|+|t-t'|\right).
\end{align*}
Thanks to \eqref{eq:regass} we have 
\begin{equation}
\label{eq:regclaim.1}
\widetilde u\in C([0,\infty)\times \R)\cap W^{1,\infty}((0,T)\times\R),\qquad T>0,
\end{equation}
and then \eqref{eq:regclaim}.
\end{proof}

The following simple example shows that we cannot expect $C^2$ regularity on the  solutions.
More precisely, we start with constant initial data and we explicitly construct conservative solutions 
exhibiting  a singularity in the second derivative. The insurgence of such singularity is due to the lack of continuity of the 
nonlinear source $\Phi'$.

\begin{example}
\label{ex:reg}
Consider the function
\begin{equation}
\label{eq:reg.1}
u(t,x)=
\begin{cases}
\sqrt{2}\sin(\sqrt{2}t),&\qquad \text{if $0\le t\le \frac{\pi}{4\sqrt{2}}$},\\
\sqrt{2} t+1-\frac{\pi}{4},&\qquad \text{if $t\ge \frac{\pi}{4\sqrt{2}}$}.
\end{cases}
\end{equation}
Clearly, $u$ solves the problem
\begin{equation*}
\begin{cases}
\ptt u=\pxx u-\Phi'(u),&\quad t>0,x\in\R,\\
u(0,x)=0,&\quad x\in\R,\\
\pt u(0,x)=2,&\quad x\in\R,
\end{cases}
\end{equation*}
but 
\begin{equation*}
u\in C^1\setminus C^2.
\end{equation*}
Indeed
\begin{align*}
\lim_{t \to \frac{\pi}{4\sqrt{2}}^- }u\left(t,x\right)=1,\qquad & \lim_{t \to \frac{\pi}{4\sqrt{2}}^+}u\left(t,x\right)=1,\\
\lim_{t \to \frac{\pi}{4\sqrt{2}}^-}\pt u\left(t,x\right)=\sqrt{2},\qquad & \lim_{t \to \frac{\pi}{4\sqrt{2}}^+} \pt u\left(t,x\right)=\sqrt{2},\\
\lim_{t \to \frac{\pi}{4\sqrt{2}}^-}\ptt u\left(t,x\right)=-2,\qquad & \lim_{t \to\frac{\pi}{4\sqrt{2}}^+ } \ptt u\left(t,x\right)=0.
\end{align*}
The energy associated to \eqref{eq:reg.1} is
\begin{equation*}
E(t)=\int_0^L\left(\frac{(\pt u_\eps(t,x))^2+(\px u_\eps(t,x))^2}{2}+ \Phi(u_\eps(t,x))\right)dx=2L.
\end{equation*}
\end{example}

\section{Discontinuities, debonding and propagation of singularities}
\label{disc}
In this section we shall focus  on some qualitative analysis aimed to investigate the occurence of singularities in the solutions
of  \eqref{eq:el} and the interplay of such singularities with debonding process. 
Based on a first order system associated to \eqref{eq:el}, we give a qualitative description of the discontinuity curves 
of the first derivatives of the solutions. These are the loci where the dissipation of energy occurs. 
Moreover, 
 we show that we cannot expct more regularity even if we
consider more regular inital data.

We can rewrite the equation in \eqref{eq:el} as a first order system in the following way
\begin{equation}
\label{eq:system}
\pt Z+A\px Z=B(Z),
\end{equation}
where
\begin{equation*}
Z=\left(\begin{matrix}z_1\\z_2\\z_3\end{matrix}\right)=\left(\begin{matrix}\pt u\\\px u\\u\end{matrix}\right),\qquad
A=\left(\begin{matrix}0&-1&0\\-1&0&0\\0&0&0\end{matrix}\right), \qquad B(Z)=\left(\begin{matrix}-\Phi'(z_3)\\0\\z_1\end{matrix}\right).
\end{equation*}

Since $Z\mapsto B(Z)$ is discontinuous the solution $Z$ of \eqref{eq:system} may develop discontinuities.
Let  $t\mapsto(t,\gamma(t))$ be a  discontinuity curve for $Z$.
Thanks to the qualitative analysis of \cite[Chapter 10]{B} $\gamma$ is locally Lipschitz continuous and
 the Rankine-Hugoniot condition \cite[Section 4.2]{B} holds
 \begin{equation*}
A\left(Z(t,\gamma(t)^+)-Z(t,\gamma(t)^-)\right)=\gamma'(t)\left(Z(t,\gamma(t)^+)-Z(t,\gamma(t)^-)\right),\qquad\text{a.e. $t$},
\end{equation*}
where
\begin{equation*}
Z(t,\gamma(t)^{\pm})=\lim_{s \to \gamma(t)^\pm} Z(t,s).
\end{equation*}
Since the eigenvalues of the matrix $A$ are $-1,\,0$ and $1$, we must have
\begin{equation*}
\gamma'(t)\in\{-1,0,1\},\qquad\text{a.e. $t$},
\end{equation*}
namely $t\mapsto(t,\gamma(t))$ is a polygonal  of the plane $(t,x)$ with slopes $-1,\,0$ and $1$.

Several remarks are needed.
In addition to the propagation velocities $1,\ -1$ of the wave equation here we have one more characteristic speed.
This feature is coherent with the one obtained in \cite{Be}. There the appearance of the stationary characteristics 
was generated by a third order hyperbolic operator  and a smooth nonlinear source term $f(u)$ in one spatial dimension.
In \cite{RR} the authors completed the picture showing that if the operator is of the second order and 
the nonlinear source term $f(u)$ is smooth we can only have two characteristic speeds. 
Here, we  are able to obtain the third characteristic speed even with
a second order wave operator because our nonlinear source term $\Phi'(u)$ is discontinuous. 

System \eqref{eq:system} admits the following entropy/entropy flux pair
\begin{equation}
\label{eq:entr}
\eta(Z)=\frac{|Z|^2}{2},\qquad q(Z)=-z_1z_2,\qquad Z=\left(\begin{matrix}z_1\\z_2\\z_3\end{matrix}\right)\in \R^3.
\end{equation}
Coherently with Definition \ref{def:sol}, the solutions of \eqref{eq:system} satisfy the following entropy inequality
\begin{equation}
\label{eq:entr-in}
\pt\eta(Z)+\px q(Z)\le \eta'(Z)B(Z),
\end{equation}
in the sense of distributions. When a shock occurs the inequality in \eqref{eq:entr-in} becomes strict. 
Indeed we consider dissipative solutions \cite{Se}.

The interplay between the propagation of singularities and debonding is described by the 
following necessary condition relating the singular points in space-time with the occurrence of attachment-debonding in the characteristic cone. 

\begin{theorem}
\label{th:sing}
Let $u$ be a dissipative solution of \eqref{eq:el} and $(t_0,x_0)\in (0,\infty)\times(0,L)$. We have that
 if $u$ is not $C^1$ in $(t_0,x_0)$ then for all $\eps>0$ there exist $(t_1,x_1),\,(t_2,x_2)\in \mathcal{T}_\eps(t_0,x_0)$ such that 
$|u(t_1,x_1)|<1<|u(t_2,x_2)|$,
where
\begin{equation*}
\mathcal{T}_\eps(t_0,x_0)=\bigcup_{\max\{t_0-\eps,0\}\le t\le t_0}\Big(\max\{0,x_0-\eps +(t-t_0)\},\min\{x_0+\eps-(t-t_0),L\}\Big).
\end{equation*}
\end{theorem}

\begin{proof}
We argue by contradiction, namely we prove that if there exists $\eps>0$ such that for all $(t,x)\in \mathcal{T}_\eps(t_0,x_0), \,|u(t,x)|\le1$ or for all $(t,x)\in \mathcal{T}_\eps(t_0,x_0),\, |u(t,x)|\ge1$ then $u$ is $C^1$ in $(t_0,x_0)$.

We can always choose $\eps$ so small such that 
\begin{equation}
\label{eq:C^1proof2}
t_0-\eps>0,\qquad 0<x_0-2\eps< x_0+2\eps<L,
\end{equation}
in this way
\begin{equation}
\label{eq:C^1proof3}
\mathcal{T}_\eps(t_0,x_0)=\bigcup_{t_0-\eps\le t\le t_0}\Big(x_0-\eps +(t-t_0),x_0+\eps-(t-t_0)\Big).
\end{equation}

Assume that 
\begin{equation*}
|u(t,x)|\le1, \qquad(t,x)\in \mathcal{T}_\eps(t_0,x_0),
\end{equation*}
and consider a $C^1$ function $\overline{\Phi}$ such that
\begin{equation*}
|u|\le1\Longrightarrow \overline{\Phi}(u)=\Phi(u).
\end{equation*}
Let $\overline{u}$ be the solution of the Cauchy problem
\begin{equation*}
\begin{cases}
\ptt \overline{u}=\pxx \overline{u}-\overline{\Phi}'(\overline{u}),&\quad t>t_0-\eps,x\in\R,\\
\overline{u}(t_0-\eps,x)={u}(t_0-\eps,x)\chi_{[x_0-2\eps,x_0+2\eps]}(x),&\quad x\in\R,\\
\pt \overline{u}(t_0-\eps,x)={u}(t_0-\eps,x)\chi_{[x_0-2\eps,x_0+2\eps]}(x),&\quad x\in\R.
\end{cases}
\end{equation*}
Due to the finite speed of propagation of the wave operator we have
\begin{equation*}
u=\overline{u}\qquad\text{in $\mathcal{T}_\eps(t_0,x_0)$}.
\end{equation*}
Using the first order reformulation \eqref{eq:system} of the equation for $\overline{u}$ we see that $\overline{u}$ is not developing any singularity in $\mathcal{T}_\eps(t_0,x_0)$.

In the case 
\begin{equation*}
|u(t,x)|\ge1, \qquad (t,x)\in \mathcal{T}_\eps(t_0,x_0),
\end{equation*}
we have only to consider a $C^1$ function $\overline{\Phi}$ such that
\begin{equation*}
|u|\ge1\Longrightarrow \overline{\Phi}(u)=\Phi(u),
\end{equation*}
and use the same argument.
\end{proof}

We conclude this Section by considering the Cauchy problem associated to \eqref{eq:system}. The motivations behind this analysis are:
\begin{itemize}
\item due to the finite speed of propagation, the Cauchy and Neumann problem share the same solution for short time and compactly supported initial data;
\item explicit formulas for the the solutions can be obtained for the Cauchy problem and those formulas do not rely on Fourier series the regularity issue is more clear.
\end{itemize}

We begin with the Cauchy problem
\begin{equation}
\label{eq:systemCauchy}
 \begin{cases}\pt Z+A\px Z=0 & {} \\
                                  Z(0,x)=Z_{0}(x) & {} 
          \end{cases} 
\end{equation}
where
\begin{equation*}
Z=\left(\begin{matrix}z_1\\z_2\\z_3\end{matrix}\right)=\left(\begin{matrix}\pt u\\\px u\\u\end{matrix}\right),\qquad
A=\left(\begin{matrix}0&-1&0\\-1&0&0\\0&0&0\end{matrix}\right)
\end{equation*}
and
\begin{equation*}
Z_0=\left(\begin{matrix}z_{1,0}\\z_{2,0}\\z_{3,0}\end{matrix}\right):  \mathbb{R} \to  \mathbb{R}^3,
\end{equation*}
is the vector of the initial conditions. 
If we diagonalize the matrix $A$ we obtain that $A=P^{-1}DP$, where
\begin{equation*}
D=\left(\begin{matrix}1&0&0\\0&0&0\\0&0&-1\end{matrix}\right), \qquad P=\left(\begin{matrix}1&0&1\\-1&0&1\\0&1&0\end{matrix}\right), \qquad P^{-1}=\frac{1}{2}\left(\begin{matrix}1&-1&0\\0&0&2\\1&1&0\end{matrix}\right).
\end{equation*}
If we define $W=PZ$ and $W_0=PZ_0$ we obtain that
\begin{equation}
\label{eq:systemCauchyW}
 \begin{cases}\pt W+D\px W=0 &{}\\
                                  W(0,x)=W_{0}(x)&{}
          \end{cases} 
\end{equation}
that can be easily solved as
\begin{equation}
\label{eq:systemCauchyw}
 \begin{cases}w_1(t,x)=w_{1,0}(x-t),&{} \\
                                  w_2(t,x)=w_{2,0}(x),&{} \\
                                  w_3(t,x)=w_{3,0}(x+t),&{}
          \end{cases} 
\end{equation}
Now, since $Z=P^{-1}W$, it results that $Z=S_t Z_0$, where
\begin{equation*}
S_t Z_0:= \frac{1}{2}\left(\begin{matrix}z_{1,0}(x-t)+z_{3,0}(x-t)+z_{1,0}(x)-z_{3,0}(x)\\2z_{2,0}(x+t)\\z_{1,0}(x-t)+z_{3,0}(x-t)-z_{1,0}(x)+z_{3,0}(x)\end{matrix}\right)
\end{equation*}
Now if we consider the Cauchy problem
\begin{equation}
\label{eq:systemCauchymod}
 \begin{cases}\pt Z+A\px Z=B(Z),&{}\\
                                  Z(0,x)=Z_{0}(x),&{}\\
      \end{cases} 
\end{equation}
where
\begin{equation*}
B(Z)=\left(\begin{matrix}-\Phi'(z_3)\\0\\z_1\end{matrix}\right)
\end{equation*}
we have that
\begin{align*}
Z(t,x)&=S_t Z_0 + \int_{0}^{t} S_{t-s}B(Z(s,x))\; ds \\
        &=S_t Z_0 + \frac{1}{2} \int_{0}^{t} \left( \begin{matrix}-\Phi'(z_3(s,x-(t-s)))+z_1(s,x-(t-s))-\Phi'(z_3(s,x))-z_1(s,x)\\ 0 \\-\Phi'(z_3(s,x-(t-s)))+z_1(s,x-(t-s))+\Phi'(z_3(s,x))+z_1(s,x)\end{matrix}\right)\; ds.
\end{align*}
The fact that the evolution of $z_2=\px u$ only involves  $S_t$ implies that the new singularities of $u$ may occur only in $\pt u(t, \cdot)$. From the physical point of view this means that the stretching $\px u$ is not sensitive to debonding-attachment phenomena.

\section{Numerical examples}
\label{sec:numerics}

In this Section we provide some numerical examples of solutions of the system described in \ref{eq:el}. As we will see, they exhibit, with proper initial conditions, a rich phenomenology. We will consider both cases with smooth and non-smooth initial conditions in order to obtain solutions with different behaviors in their derivative that explicitly show propagation along the characteristics. 

\begin{example}
\label{ex:1.1}
Let us take initial value such that $u(0,x)=u_0(x)$ is in $C^2([0,L])$:
\begin{align}
\label{eq:initC2}
\begin{cases}
u_0(x)=\xi_0 \left(\frac{x^3}{3}-L\frac{x^2}{2}\right),&\quad 0\le x \le L,\\
u_1(x)=\xi_1,&\quad 0 \le x \le L.
\end{cases}
\end{align}
\end{example}
\noindent In Figures \ref{fig:uregolare} and \ref{fig:uregolareder} we plot, respectively, the solution of \ref{eq:el} with initial conditions (\ref{eq:initC2}) and its first order derivatives with respect to time and space coordinates. In the simulation we have used the values $\xi_0=0.006$, $\xi_1=1.2$, $L=10$ and a total time evolution $T=10$. In Figure \ref{fig:uregolare} we have also inserted the plans with $u=\pm 1$ in order to show the regions where the solution has reached and exceeded the critical values. From the results, it is evident that the initial conditions allow a part of the system to pass $u=1$. It is possible to see that these two values depend on $\xi_0$ and $\xi_1$. On the other hand, the simulations show that the system also exhibits another feature: the debonding process is reversed and $u$ takes values less than $1$. The values $t^*$ (time) and $x^*$ (position) where again $u=1$ and the debonding is reversed are more evident by inspecting the behavior of $\pt u$ and $\px u$. Interestingly, $\px u$ is not sensible to the debonding process. This is consistent with the results discussed at the end of Section \ref{disc}. Moreover, it is evident that the inversion point act as a source for the explicit observation of the propagation along the characteristic curves. These features will appear also in other examples in the following.

\begin{figure}[t]
\begin{center}
\includegraphics[height=0.4\textwidth]{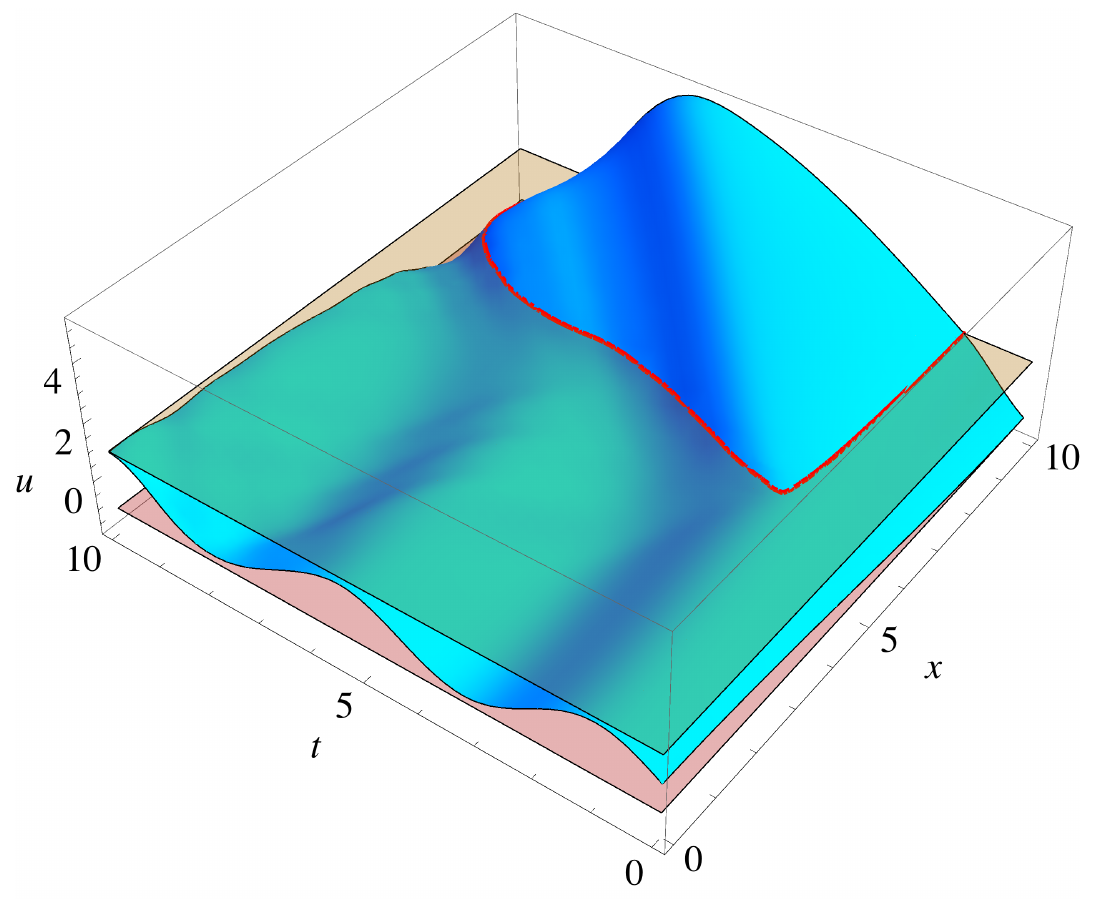}
\caption{(Color online) Solution of (\ref{eq:el}) with initial conditions (\ref{eq:initC2}). See text for the numerical values used in the simulation. 
}
\label{fig:uregolare}
\end{center}
\end{figure}

\begin{figure}[t]
\begin{center}
\begin{tabular}{c c}
\includegraphics[height=0.4\textwidth]{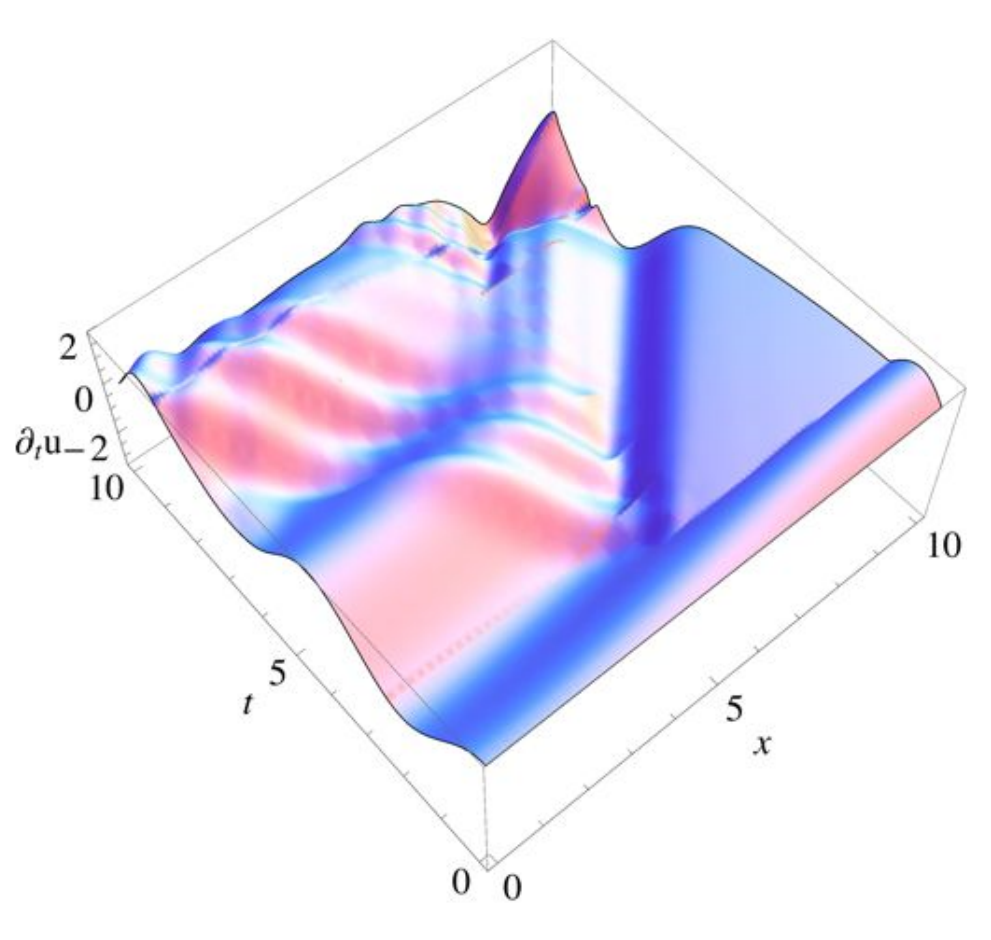} & \includegraphics[height=0.4\textwidth]{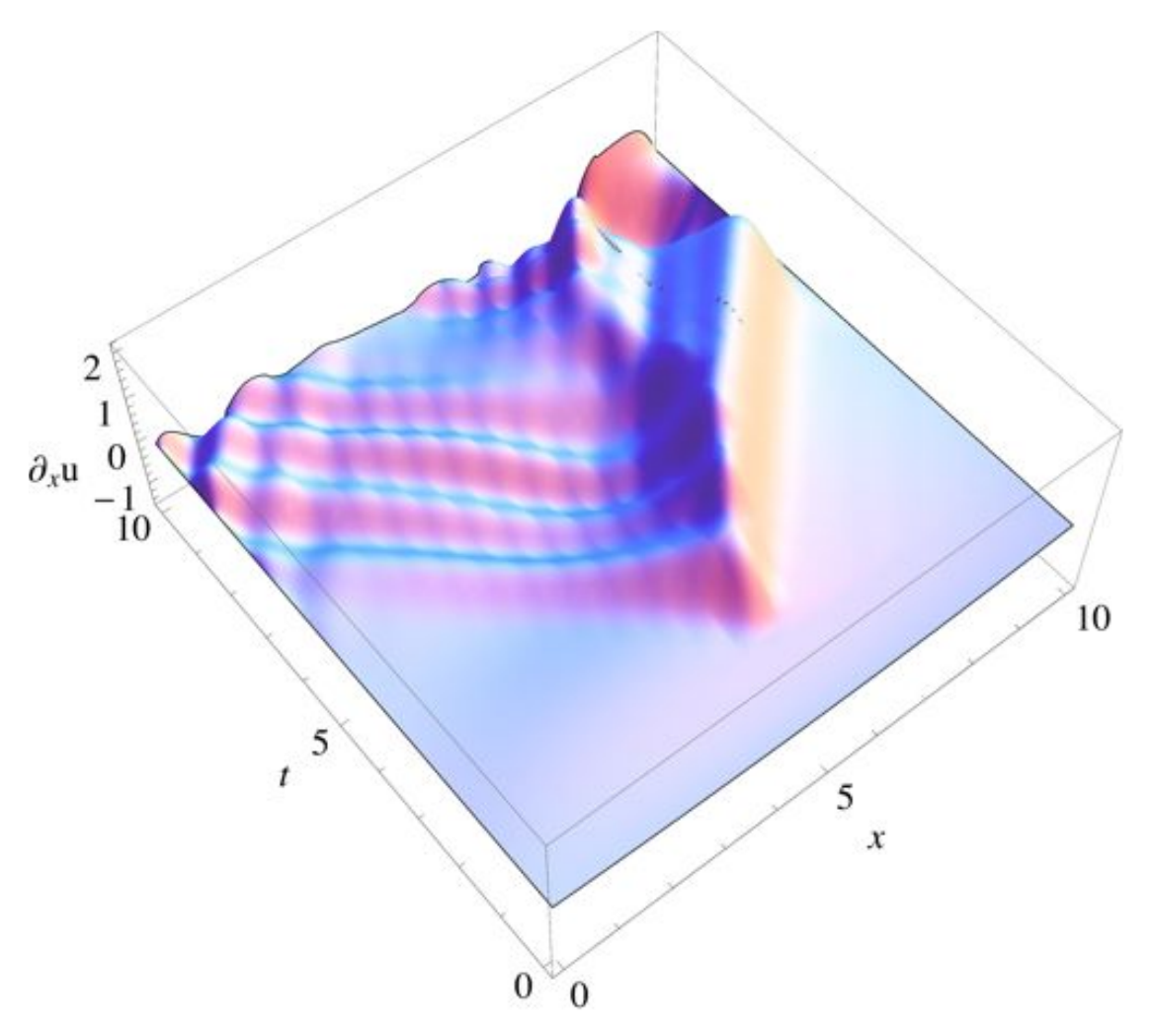}
\end{tabular}
\caption{(Color online) Left: derivative with respect to time of $u$ shown in Figure \ref{fig:uregolare}. Right: derivative with respect to space of $u$ shown in Figure \ref{fig:uregolare}. See text for the numerical values used in the simulation.
}
\label{fig:uregolareder}
\end{center}
\end{figure}

\begin{example}
\label{ex:1.2}
We consider initial values such that $u$ is in $C^1([0,L])$ for $t=0$:
\begin{align}
\label{eq:initC1}
\begin{cases}
u_0(x)=\xi_0 b\; c(x),&\quad 0\le x \le L,\\
u_1(x)=\xi_1,&\quad 0 \le x \le L.
\end{cases}
\end{align}
where $b$ is a constant defined as
\begin{equation}
\label{eq:initC1funca}
b=\frac{1}{-\frac{1}{4} L \sqrt{4 a^2-L^2}+a^2
   \left(-\tan ^{-1}\left(\frac{L}{\sqrt{4
   a^2-L^2}}\right)\right)+a L},
\end{equation}
and
\begin{align}
\label{eq:initC1funcb1}
c(x)=&-\frac{1}{2} x \sqrt{a^2-x^2}+\frac{1}{2} a^2 \tan
   ^{-1}\left(\frac{x
   \sqrt{a^2-x^2}}{x^2-a^2}\right)+a x ,\quad 0\le x \le L/2,\\
\label{eq:initC1funcb2}
c(x)=&-\frac{1}{4} L \sqrt{4 a^2-L^2}+a^2 \left(-\tan
   ^{-1}\left(\frac{L}{\sqrt{4
   a^2-L^2}}\right)\right)\notag\\
   &+\frac{1}{2} \left((L-x)
   \sqrt{a^2-(L-x)^2}+a^2 \tan
   ^{-1}\left(\frac{L-x}{\sqrt{a^2-(L-x)^2}}\right)+2
   a x\right),\quad L/2 < x \le L,
\end{align}
with $a>L/2$.
\end{example}
\noindent We notice that the second space derivative of $c$ is not continuous in $x=L/2$.
We have performed two different simulations. In Figures \ref{fig:uC1} and \ref{fig:uC1der} we plot, respectively, the solution of \eqref{eq:el} with initial conditions (\ref{eq:initC1}) and its first order derivatives with respect to time and space coordinates. In the simulation we have used the values $\xi_0=0.7$, $\xi_1=1.1$, $a=6$, $L=10$ and a total time evolution $T=3$. In Figures \ref{fig:uC1b} and \ref{fig:uC1bder} we find the solution for the same system but with $\xi_1=1.4$. As in the previous example, we observe that the inversion point (where the debonding is reversed) acts as a source for the direct observation of propagation along the characteristics. Moreover, we explicitly observe propagation along the characteristics from the initial point $(t=0,x=L/2)$ where $\px u$ is not $C^2$. From other numerical experiments (not shown in the paper) we can deduce that this phenomenon is ubiquitous whenever there are two (or more) points in the $x$-domain at $t=0$ with $\px u$ not in $C^2$.

\begin{figure}[t]
\begin{center}
\includegraphics[height=0.4\textwidth]{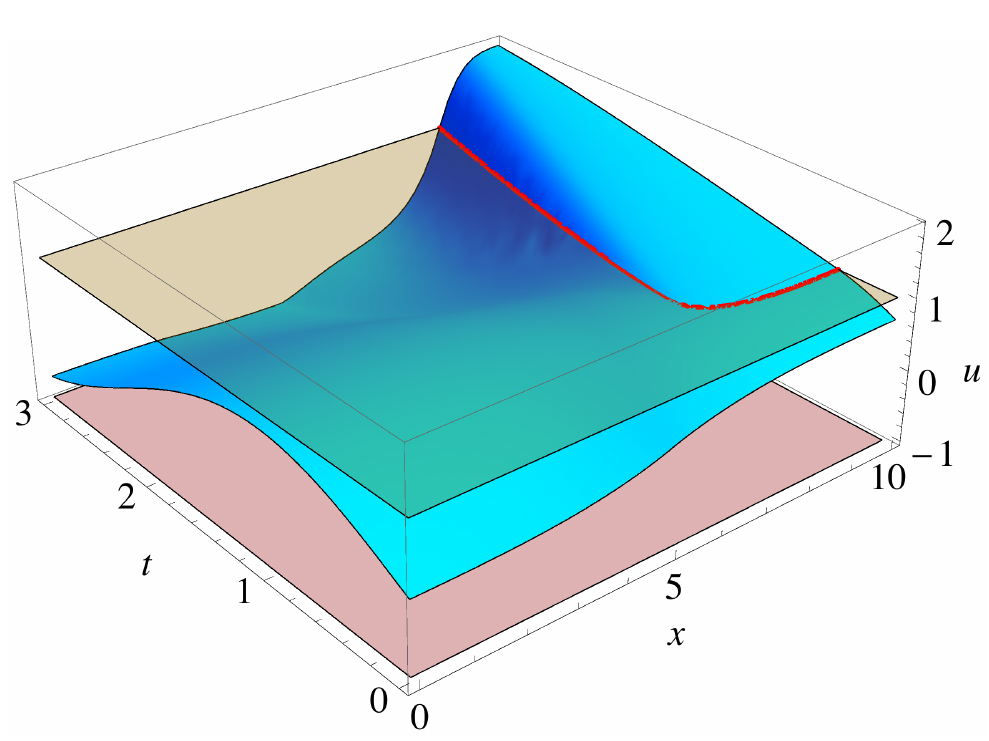}
\caption{(Color online) Solution of (\ref{eq:el}) with initial conditions (\ref{eq:initC1}). See text for the numerical values used in the simulation. 
}
\label{fig:uC1}
\end{center}
\end{figure}

\begin{figure}[t]
\begin{center}
\begin{tabular}{c c}
\includegraphics[height=0.4\textwidth]{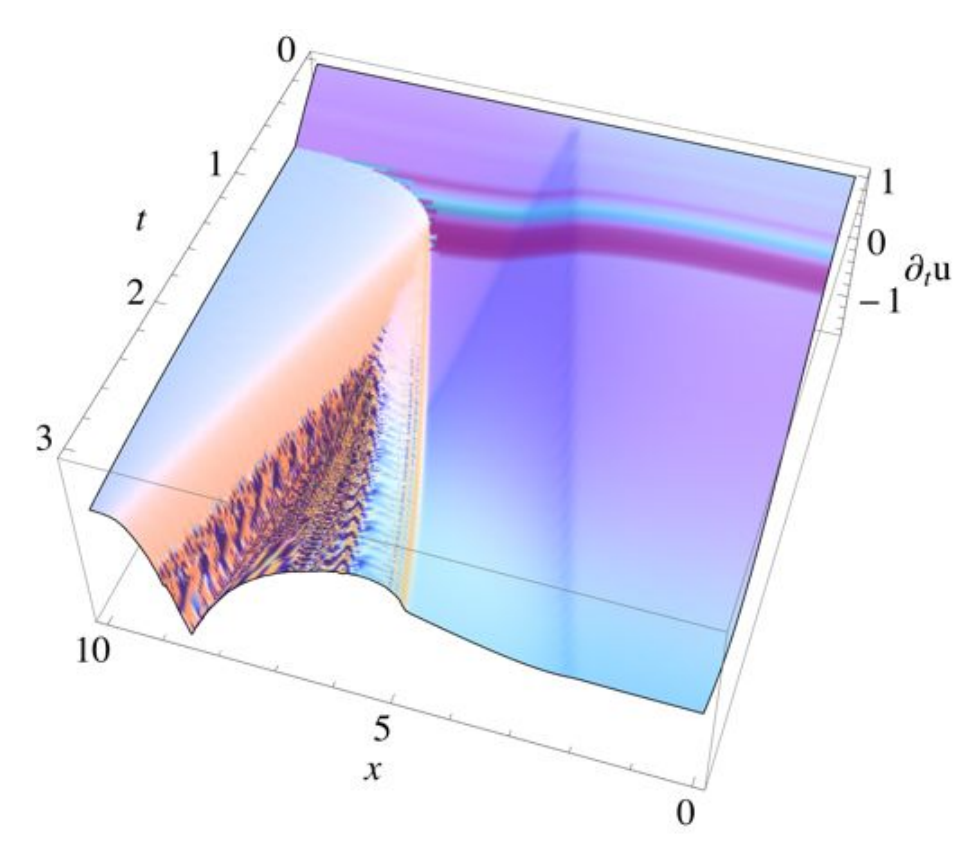} & \includegraphics[height=0.4\textwidth]{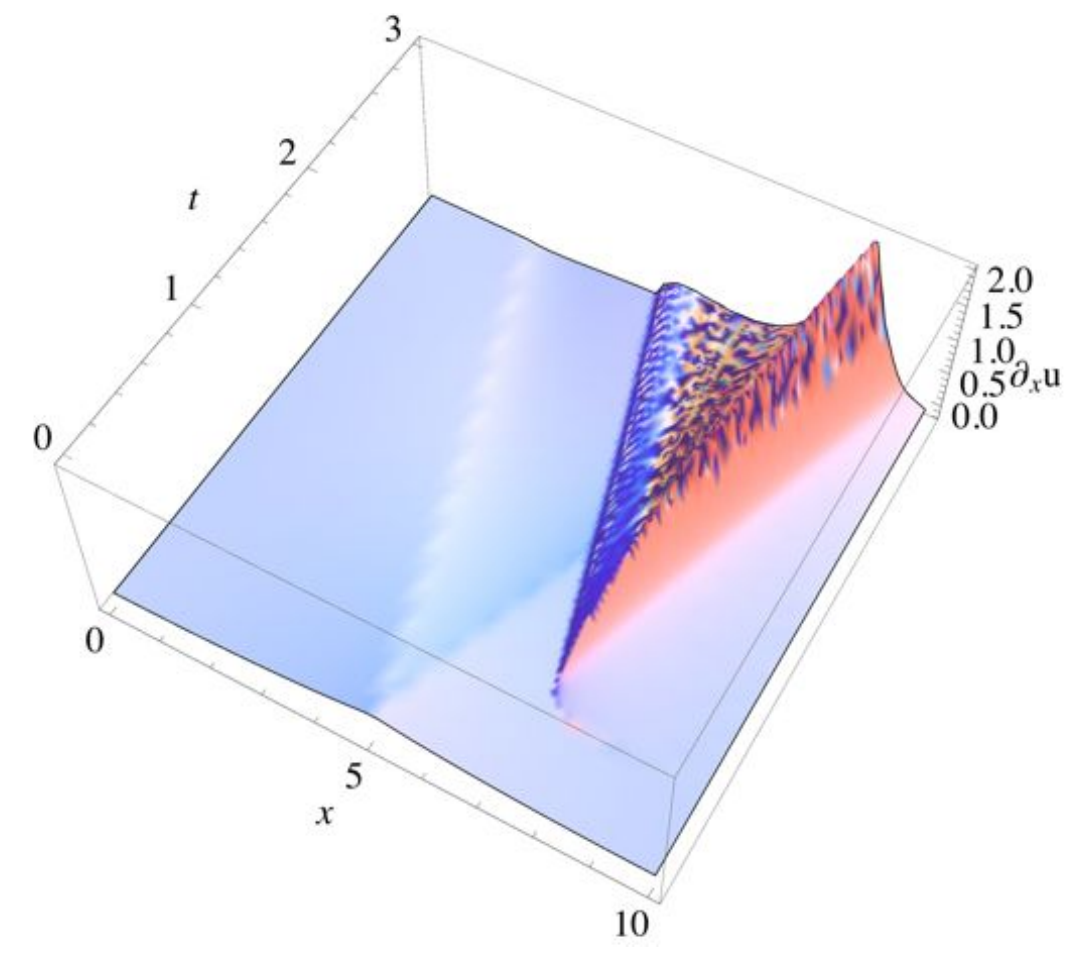}
\end{tabular}
\caption{(Color online) Left: derivative with respect to time of $u$ shown in Figure \ref{fig:uC1}. Right: derivative with respect to space of $u$ shown in Figure \ref{fig:uC1}. See text for the numerical values used in the simulation.
}
\label{fig:uC1der}
\end{center}
\end{figure}

\begin{figure}[t]
\begin{center}
\includegraphics[height=0.4\textwidth]{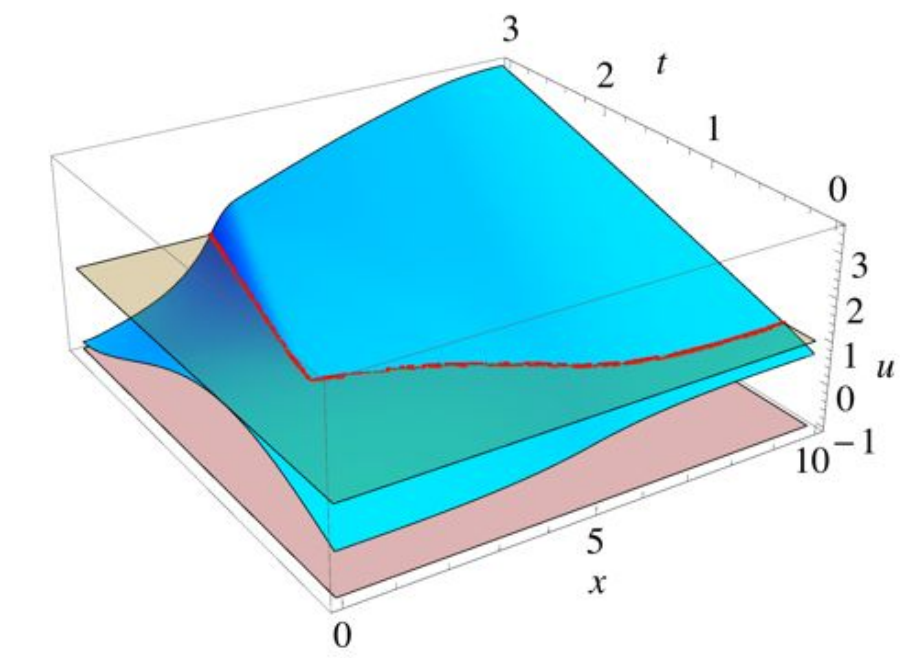}
\caption{(Color online) Solution of (\ref{eq:el}) with initial conditions (\ref{eq:initC1}). See text for the numerical values used in the simulation. 
}
\label{fig:uC1b}
\end{center}
\end{figure}

\begin{figure}[t]
\begin{center}
\begin{tabular}{c c}
\includegraphics[height=0.4\textwidth]{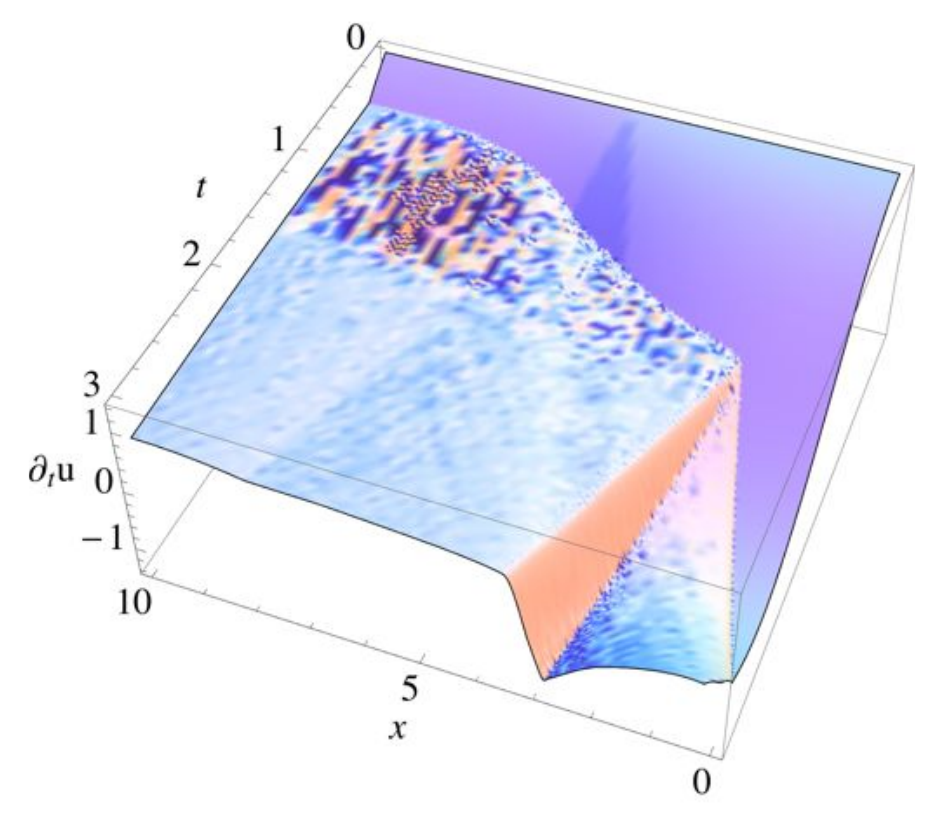} & \includegraphics[height=0.4\textwidth]{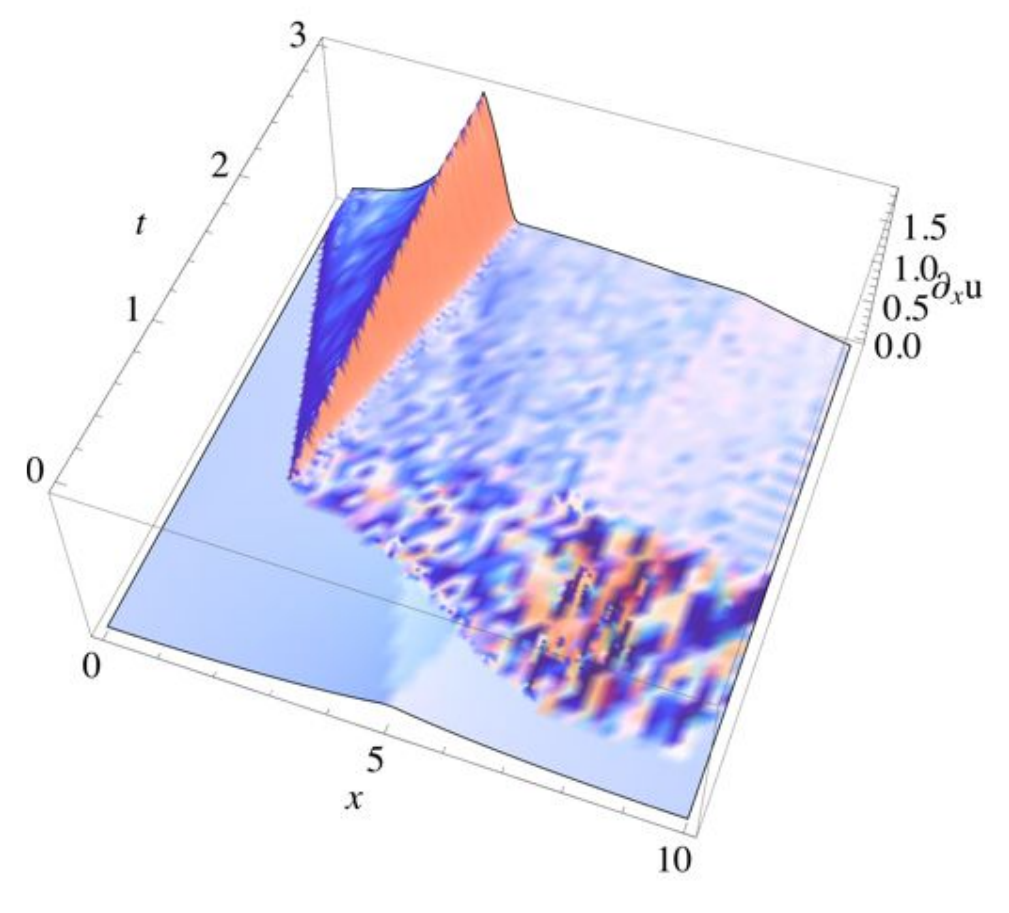}
\end{tabular}
\caption{(Color online) Left: derivative with respect to time of $u$ shown in Figure \ref{fig:uC1b}. Right: derivative with respect to space of $u$ shown in Figure \ref{fig:uC1b}. See text for the numerical values used in the simulation.
}
\label{fig:uC1bder}
\end{center}
\end{figure}

\begin{example}
\label{ex:3.1}
We consider now the initial conditions:
\begin{align}
\label{eq:initC1middle}
\begin{cases}
u_0(x)=\xi_0 (-\frac{1}{2}+b\; c(x))+1,&\quad 0\le x \le L,\\
u_1(x)=\xi_1,&\quad 0 \le x \le L.
\end{cases}
\end{align}
where we have defined $b$ and $c$ in Equations (\ref{eq:initC1funca})-(\ref{eq:initC1funcb2}). 
\end{example}
\noindent Thus, we fix the discontinuity point of the second derivative of $u$ with respect to $x$ (at $t=0$) when $u(0,L/2)=1$. Moreover, we set $\xi_0=0.7$, $\xi_1=-1.2$ (the initial velocity is reversed), $a=6$, $L=10$ and a total time evolution $T=3$. As in the previous cases, in Figures \ref{fig:uC1middle} and \ref{fig:uC1middleder} we plot, respectively, the solution of (\ref{eq:el}) with initial conditions (\ref{eq:initC1middle}) and its first order derivatives with respect to time and space coordinates. We notice that the point $(t=0, x=L/2)$ is now both an inversion point and a discontinuity point for the initial second space derivative of $u$. Thus we again observe the explicit propagation along the characteristics.  Moreover, the value of $\xi_1$ is large enough to observe a complete debonding phenomenon with $u<-1$ after some time. This is evident form the behavior of $\pt u$ in Figure \ref{eq:initC1funcb2}. We also notice that there are not new sources of characteristics.  

\begin{figure}[t]
\begin{center}
\includegraphics[height=0.4\textwidth]{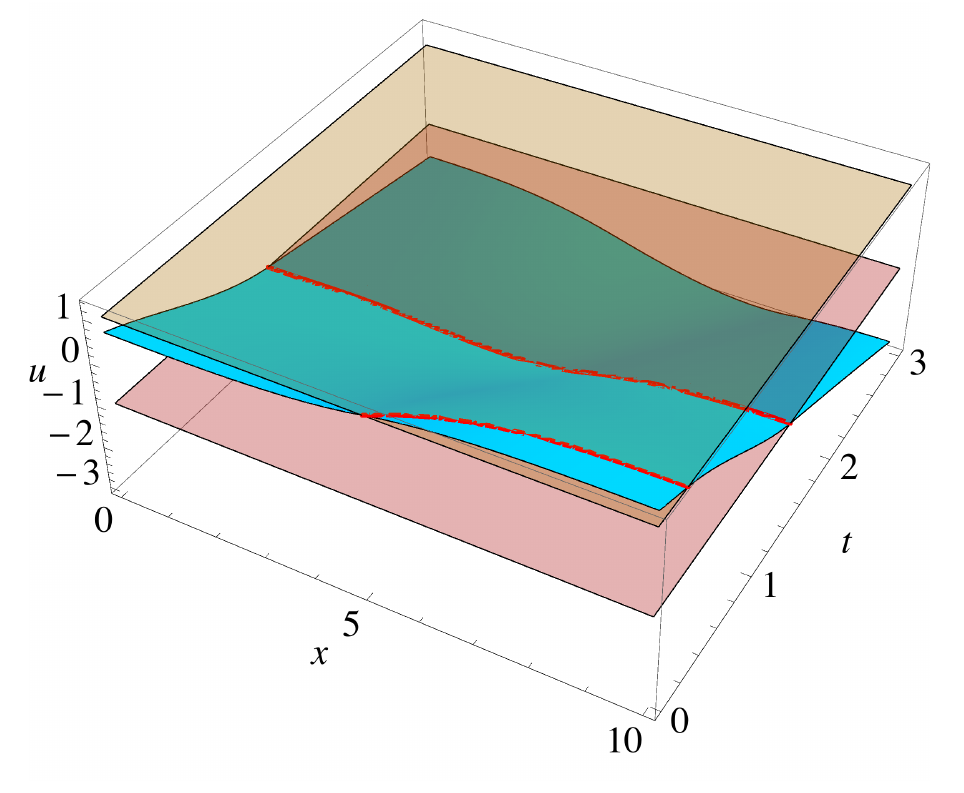}
\caption{(Color online) Solution of (\ref{eq:el}) with initial conditions (\ref{eq:initC1middle}). See text for the numerical values used in the simulation. 
}
\label{fig:uC1middle}
\end{center}
\end{figure}

\begin{figure}[t]
\begin{center}
\begin{tabular}{c c}
\includegraphics[height=0.4\textwidth]{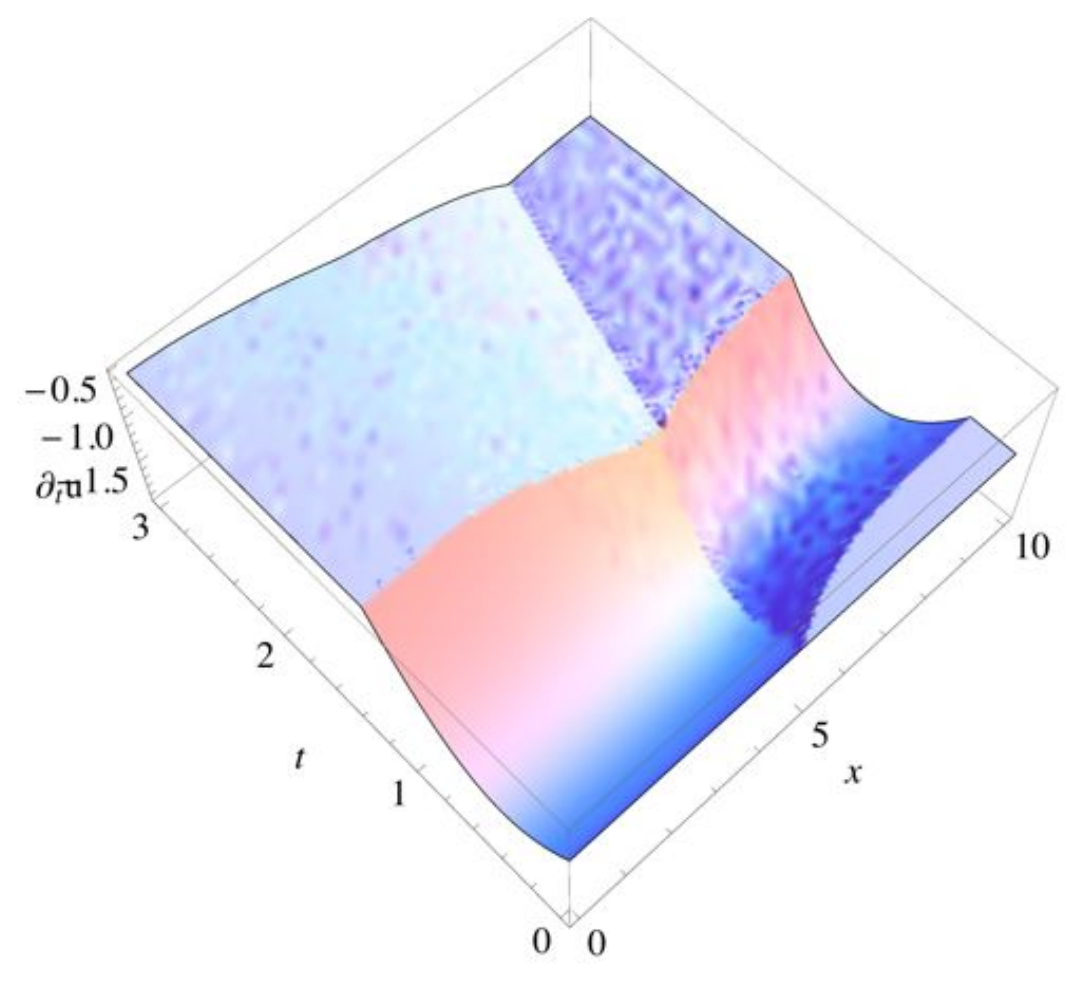} & \includegraphics[height=0.4\textwidth]{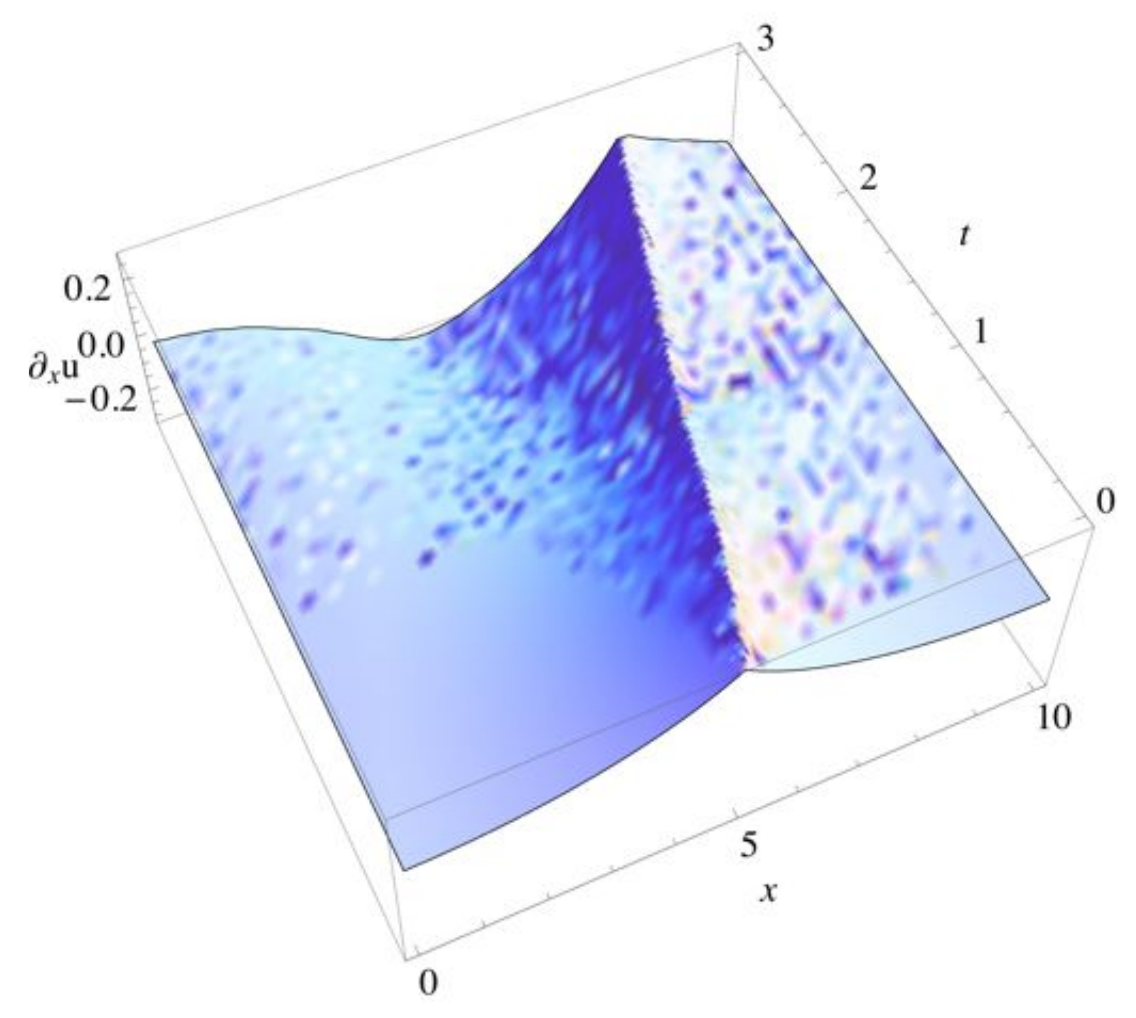}
\end{tabular}
\caption{(Color online) Left: derivative with respect to time of $u$ shown in Figure \ref{fig:uC1middle}. Right: derivative with respect to space of $u$ shown in Figure \ref{fig:uC1middle}. See text for the numerical values used in the simulation.
}
\label{fig:uC1middleder}
\end{center}
\end{figure}

We stress that we have numerically tested that separating the inversion point and the discontinuity point gives rise to two separated characteristics sets. This is evident from the results obtained in the following example.
\begin{example}
\label{ex:4}
Let us consider the initial conditions
\begin{align}
\label{eq:initC1middleb}
\begin{cases}
u_0(x)=\xi_0 (\frac{1}{2}+b\; c(x)),&\quad 0\le x \le L,\\
u_1(x)=\xi_1,&\quad 0 \le x \le L.
\end{cases}
\end{align}
\end{example}
\noindent In Figures \ref{fig:uC1middleb} and \ref{fig:uC1middlebder} we show, respectively, the solution of (\ref{eq:el}) with initial conditions (\ref{eq:initC1middleb}) and its first order derivatives with respect to time and space coordinates [same parameters used for the simulations with the conditions in Equation (\ref{eq:initC1middle})]. Moreover, in this case there is not a complete debonding, even in the example where a large negative value of $\xi_1$ has been chosen. As a consequence, we can observe a new set of characteristics in the plots of $\pt u$ and $\px u$.

\begin{figure}[t]
\begin{center}
\includegraphics[height=0.4\textwidth]{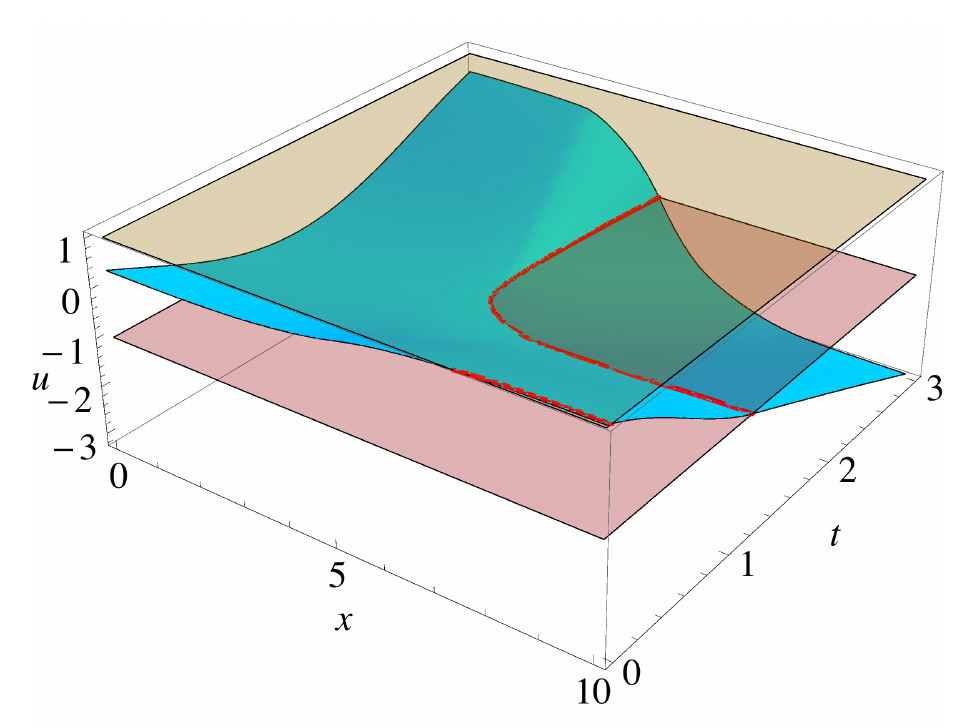}
\caption{(Color online) Solution of (\ref{eq:el}) with initial conditions (\ref{eq:initC1middleb}). See text for the numerical values used in the simulation. 
}
\label{fig:uC1middleb}
\end{center}
\end{figure}

\begin{figure}[t]
\begin{center}
\begin{tabular}{c c}
\includegraphics[height=0.4\textwidth]{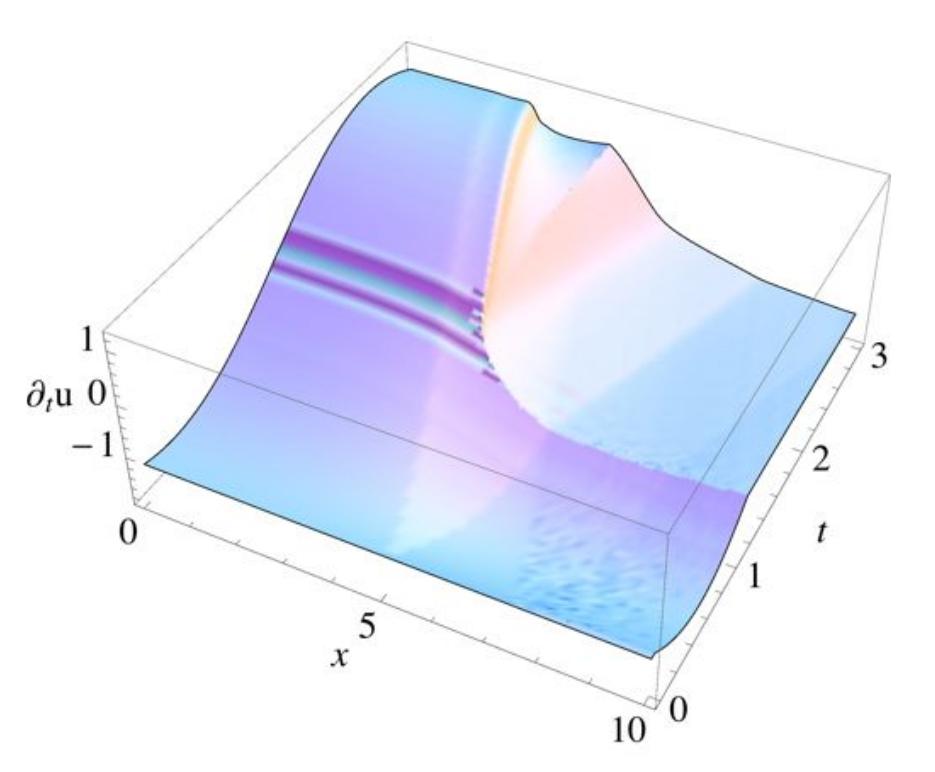} & \includegraphics[height=0.4\textwidth]{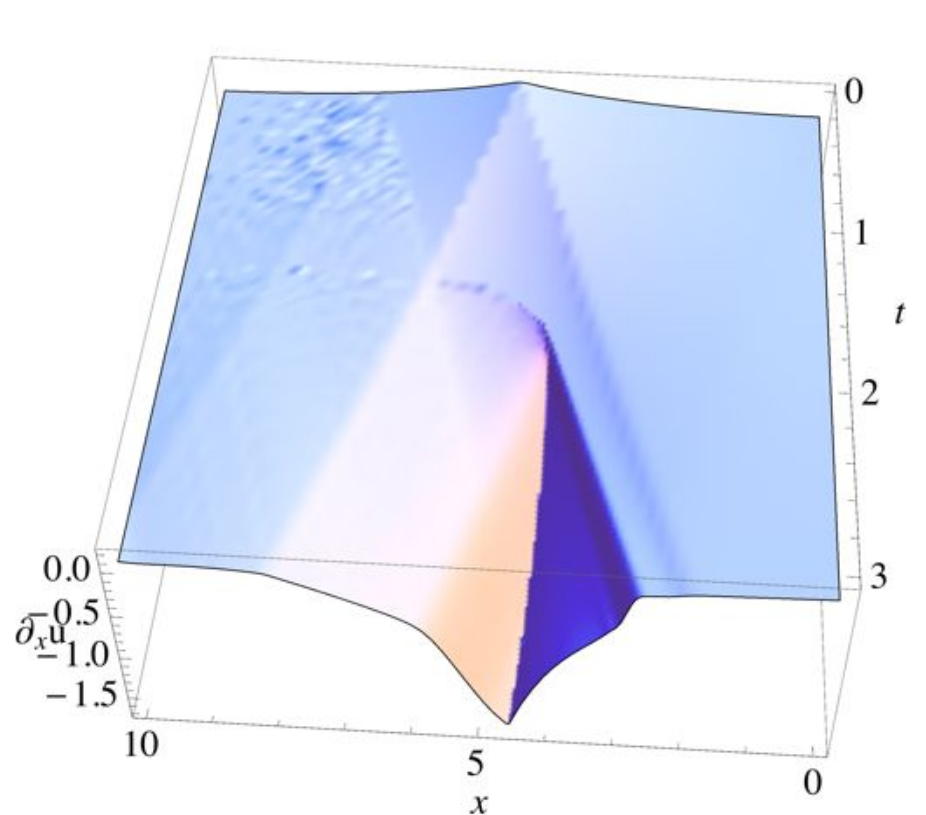}
\end{tabular}
\caption{(Color online) Left: derivative with respect to time of $u$ shown in Figure \ref{fig:uC1middleb}. Right: derivative with respect to space of $u$ shown in Figure \ref{fig:uC1middle}. See text for the numerical values used in the simulation.
}
\label{fig:uC1middlebder}
\end{center}
\end{figure}

\begin{figure}[t]
\begin{center}
\includegraphics[height=0.4\textwidth]{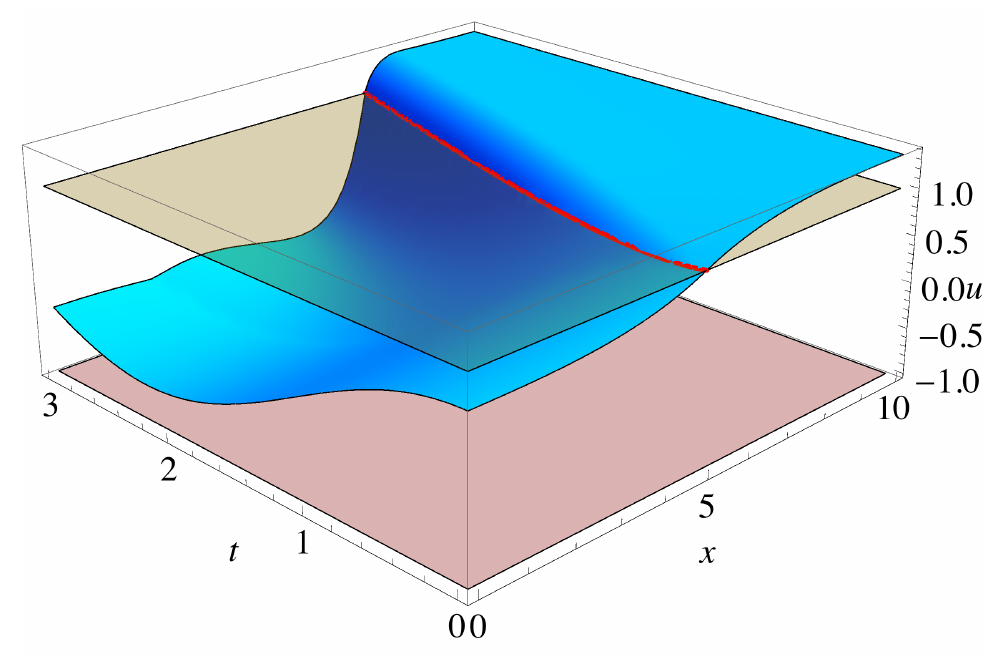}
\caption{(Color online) Solution of (\ref{eq:el}) with initial conditions (\ref{eq:initC1double}). See text for the numerical values used in the simulation. 
}
\label{fig:uC1double}
\end{center}
\end{figure}

\begin{figure}[t]
\begin{center}
\begin{tabular}{c c}
\includegraphics[height=0.4\textwidth]{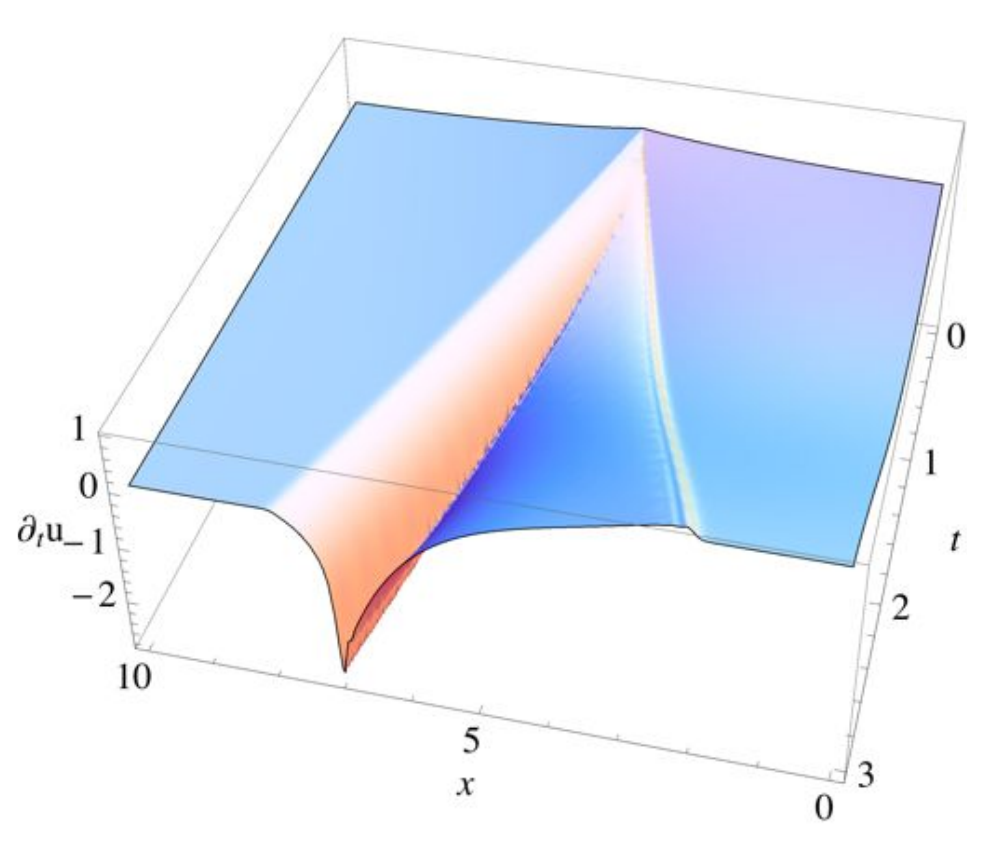} & \includegraphics[height=0.4\textwidth]{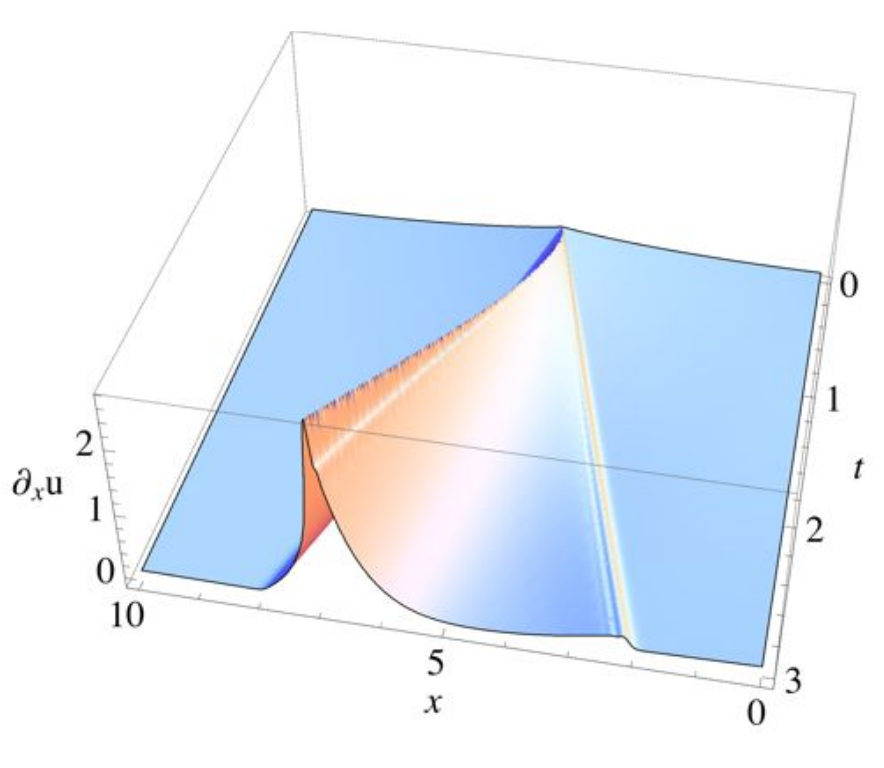}
\end{tabular}
\caption{(Color online) Left: derivative with respect to time of $u$ shown in Figure \ref{fig:uC1double}. Right: derivative with respect to space of $u$ shown in Figure \ref{fig:uC1double}. See text for the numerical values used in the simulation.
}
\label{fig:uC1doubleder}
\end{center}
\end{figure}

\begin{example}
\label{ex:5}
We have also considered the  case of initial condition where $\pt u(0,x)=u_1(x)$ is in $C([0,L])$:
\begin{align}
\label{eq:initC1double}
\begin{cases}
u_0(x)=\xi_0 (-\frac{1}{2}+b\; c(x))+1,&\quad 0\le x \le L,\\
u_1(x)=\xi_1 b\;\frac{d}{dx}[c(x)],&\quad 0 \le x \le L.
\end{cases}
\end{align}
\end{example}
The discontinuity point of the second derivative of $u$ with respect to $x$ (at $t=0$) appears for $x=L/2$ when $u(0,L/2)=1$. We have fixed $\xi_0=0.7$, $\xi_1=0.8$, $a=6$, $L=10$ and a total time evolution $T=3$. In Figures \ref{fig:uC1double} and \ref{fig:uC1doubleder} we show, respectively, the solution of (\ref{eq:el}) with initial conditions (\ref{eq:initC1double}) and its first order derivatives with respect to time and space coordinates. We observe the propagation of the discontiuity in both partial derivatives.

\begin{figure}[t]
\begin{center}
\includegraphics[height=0.4\textwidth]{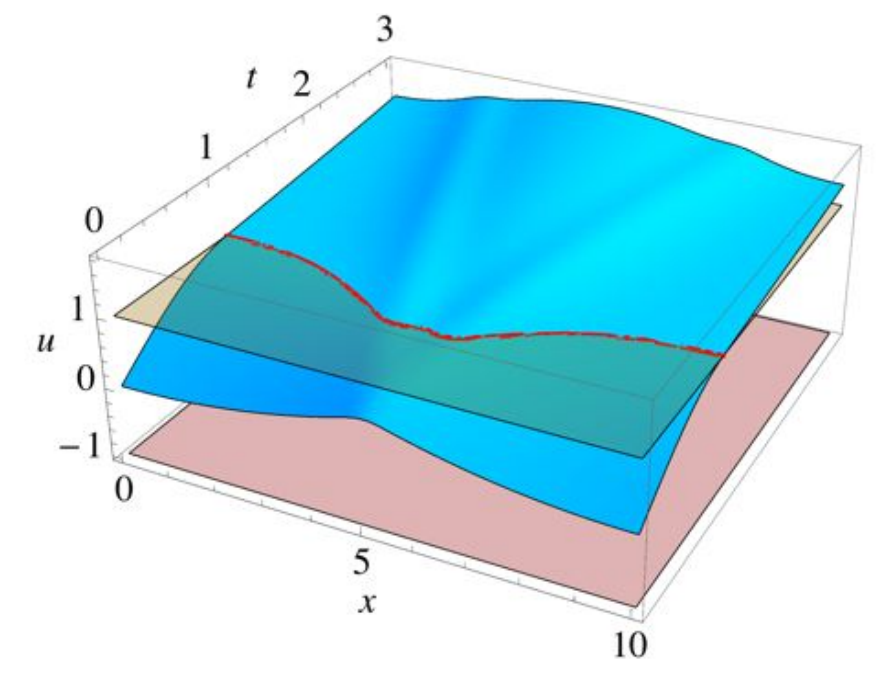}
\caption{(Color online) Solution of (\ref{eq:el}) with initial conditions (\ref{eq:initCmoll}). See text for the numerical values used in the simulation. 
}
\label{fig:uCmoll}
\end{center}
\end{figure}

\begin{figure}[t]
\begin{center}
\begin{tabular}{c c}
\includegraphics[height=0.4\textwidth]{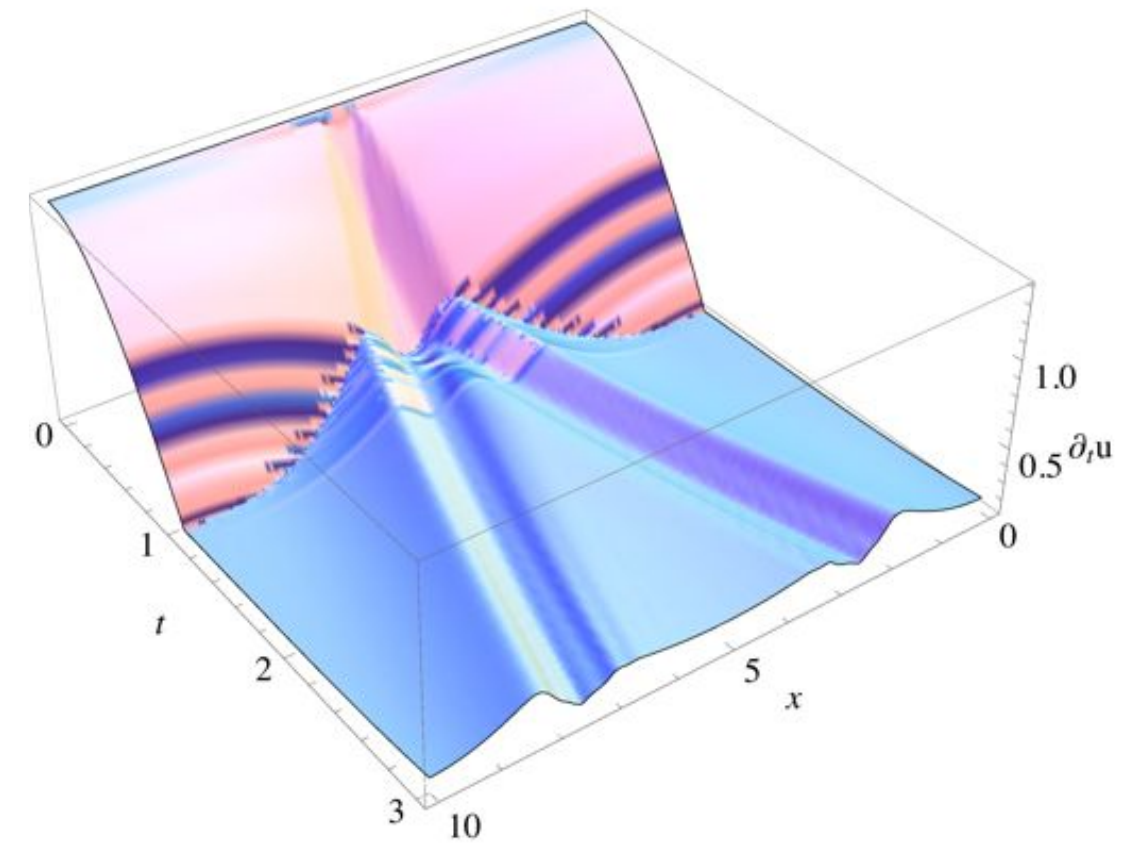} & \includegraphics[height=0.4\textwidth]{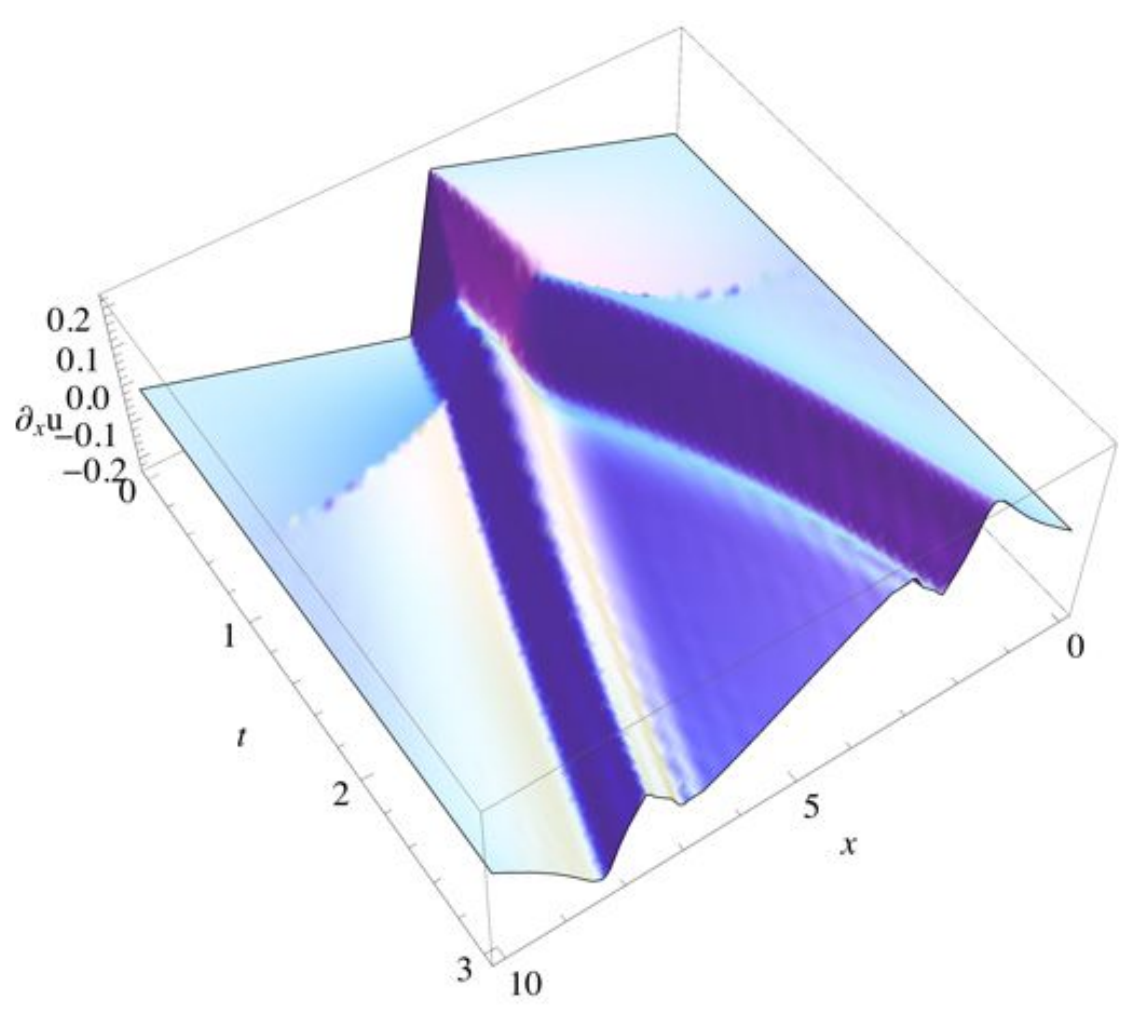}
\end{tabular}
\caption{(Color online) Left: derivative with respect to time of $u$ shown in Figure \ref{fig:uCmoll}. Right: derivative with respect to space of $u$ shown in Figure \ref{fig:uCmoll}. See text for the numerical values used in the simulation.
}
\label{fig:uCmollder}
\end{center}
\end{figure}

\begin{example}
\label{ex:1.3}
Finally, we have considered the case of initial condition where $u$ at $t=0$ is in $C([0,L])$. In order to perform the numerical simulation we have mollified the initial data in the following way:
\begin{align}
\label{eq:initCmoll}
\begin{cases}
u_0(x)=\xi_0 f_\eta(x),&\quad 0\le x \le L,\\
u_1(x)=\xi_1,&\quad 0 \le x \le L,
\end{cases}
\end{align}
where
\begin{align}
\label{eq:feta}
f_\eta(x)=\frac{2}{\left(\frac{L}{2}-\eta\right) L}\times
\begin{cases}
x^2,&\quad 0\leq x<\frac{L}{2}-\eta,\\
-\frac{L/2-\eta}{\eta}x^2+L\frac{L/2-\eta}{\eta}x-L\frac{(L/2-\eta)^2}{2\eta},&\frac{L}{2}-\eta\leq x<\frac{L}{2}+\eta,\\
(x-L)^2,&\frac{L}{2}+\eta\le x\le L,
\end{cases}
\end{align}
and $\eta$ is the mollification parameter.
\end{example}
\noindent We have fixed $\xi_0=0.5$, $\xi_1=1.4$, $\eta=0.3$, $L=10$ and a total time evolution $T=3$. In Figures \ref{fig:uCmoll} and \ref{fig:uCmollder} we show, respectively, the solution of (\ref{eq:el}) with initial conditions (\ref{eq:initCmoll}) and its first order derivatives with respect to time and space coordinates. Also in this case, we directly observe the presence of characteristics due to debonding and an hint of the propagation due to the jump of the value of the first derivative in $t=0, x=L/2$. We have verified an analogous behavior of the solution $u$ when the value of $\eta$ is reduced.



\begin{thebibliography}{10}
\bibitem{Be}
M. Beals, \textit{Self-spreading and strength of singularities for solutions to semilinear wave equations},  Ann. of Math. (2), \textbf{118} (1983), 187-214.

\bibitem{B}
A. Bressan.
\newblock \textit{Hyperbolic Systems of Conservation Laws. The one-dimensional Cauchy problem.}
\newblock  Oxford Lecture Series in Mathematics and its Applications, vol. 20, Oxford University Press, Oxford, 2000.

\bibitem{DBQ} P.G. de Gennes, F. Brochard-Wyart, D. Qu\'{e}r\'{e},
\newblock \textit{Capillarity and wetting phenomena: drops, pearls and waves.}
\newblock Springer-Verlag, New York, 2004

\bibitem{KKR} K. Kendall, M. Kendall, F. Rehfeldt,
\newblock \textit{Adhesion of Cells, Viruses and Nanoparticles},
\newblock Springer-Verlag, New York, 2011.

\bibitem{MP} F. Maddalena, D. Percivale, \textit{Variational models for peeling problems},
 Int. Free Boundaries, \textbf{10} (2008), 503-516.

\bibitem{MPPT} F. Maddalena, D. Percivale, G. Puglisi, L. Truskinowsky,
\textit{Mechanics of reversible unzipping},
Continuum Mech. Thermodyn., \textbf{21} (2009), 251-268.

  \bibitem{MPT} F. Maddalena, D. Percivale, F. Tomarelli, \textit{Adhesive flexible material structures}, 
  Discr. Continuous Dynamic. Systems B, \textbf{17} (2012), 553-574.

 \bibitem{MPT1} F. Maddalena, D. Percivale, F. Tomarelli, \textit{Elastic structures in adhesion interaction},  Variational Analysis and Aerospace Engineering ,
 Editors: A.Frediani, G.Buttazzo,  Ser. Springer  Optimization and its Applications, Vol. 66, ISBN 978-1-4614-2434-5, (2012) 289-304.

\bibitem{taylor} M. E. Taylor, \textit{Partial Differential Equations I: Basic Theory (2nd ed.)}, Springer, 2011. 

\bibitem{RR} J. Rauch, M. C. Reed, \textit{Propagation of singularities for semilinear hyperbolic equations in one space variable}, Ann. of Math. (2), tetxbf{111} (1980), 531-552.

\bibitem{S}
J.~Simon, \textit{Compact sets in the space $L\sp p(0,T;B)$}, Ann. Mat. Pura Appl., \textbf{146} (1987), 65-96.

\bibitem{Se}
D.~Serre.
\newblock Systems of conservation laws with dissipation
\newblock {\em Lectures Notes SISSA}, 2007.
\end{thebibliography}
\end{document}